\documentclass[reqno,12pt]{amsart}

\usepackage{psfig}

\usepackage{graphicx}
\usepackage{amssymb,amsmath,amstext}
\usepackage{latexsym}

\usepackage[norelsize]{algorithm2e}

\renewcommand{\algorithmcfname}{ALGORITHM}

\newtheorem{theorem}{Theorem}

\begin{document}

\title[On Solving L-SR1 Trust-Region Subproblems]
      {On Solving L-SR1 Trust-Region Subproblems}


\author[J. Brust]{Johannes Brust}
\email{jbrust@ucmerced.edu}
\address{Applied Mathematics, University of California, Merced, Merced, CA 95343}

\author[J. Erway]{Jennifer B. Erway}
\email{erwayjb@wfu.edu}
\address{Department of Mathematics, Wake Forest University, Winston-Salem, NC 27109}

\author[R. Marcia] {Roummel F. Marcia}
\email{rmarcia@ucmerced.edu}
\address{Applied Mathematics, University of California, Merced, Merced, CA 95343}

\thanks{J.~B. Erway is supported in part by National Science Foundation grant
CMMI-1334042.}
\thanks{R.~F. Marcia is supported in part by National Science Foundation grant
CMMI-1333326.}

\date{\today}


\begin{abstract}
  In this article, we consider solvers for large-scale trust-region
  subproblems when the quadratic model is defined by a limited-memory
  symmetric rank-one (L-SR1) quasi-Newton matrix.  We propose a solver that
  exploits the compact representation of L-SR1 matrices.  Our approach
  makes use of both an orthonormal basis for the eigenspace of the L-SR1
  matrix and the Sherman-Morrison-Woodbury formula to compute global
  solutions to trust-region subproblems.   
  To compute the optimal Lagrange multiplier for the trust-region
  constraint, we use Newton's method with a judicious initial guess that
  does not require safeguarding.  A crucial property of this solver
  is that it is able to compute high-accuracy solutions even in the
  so-called \emph{hard case}.  Additionally, the optimal solution is
  determined directly by formula, not iteratively.  Numerical experiments
  demonstrate the effectiveness of this solver.
\end{abstract}

\maketitle

\newcommand{\mgap}{\;\;}
\newcommand{\bgap}{\;\;\;}
\newcommand{\qDef}{{\mathcal Q}}
\newcommand{\defined}{\mathop{\,{\scriptstyle\stackrel{\triangle}{=}}}\,}
\newcommand{\diag}{\text{diag}}
\renewcommand{\algorithmcfname}{ALGORITHM}
\makeatletter
\def\BFGS{{\small BFGS}}
\def\LBFGS{{\small L-BFGS}}
\def\LSR{{\small L-SR1}}
\def\SR{{\small SR1}}
\def\CG{{\small CG}}
\def\DFP{{\small DFP}}
\def\OBS{{\small OBS}}
\def\SSM{{\small SSM}}
\def\MSS{{\small MSS}}
\def\LSTRS{{\small LSTRS}}
\def\MATLAB{{\small MATLAB}}

\newcommand{\minimize}[1]{{\displaystyle\minim_{#1}}}
\newcommand{\minim}{\mathop{\operator@font{minimize}}}
\newcommand{\subject}{\mathop{\operator@font{subject\ to}}}  
\newcommand{\words}[1]{\mgap\text{#1}\mgap}
\makeatother

\pagestyle{myheadings}
\thispagestyle{plain}
\markboth{J. BRUST, J. B. ERWAY, AND R. F. MARCIA}{ON SOLVING L-SR1 TRUST-REGION SUBPROBLEMS}

\section{Introduction} \label{intro}
In this article, we describe a method for minimizing a
quadratic function defined by a limited-memory symmetric rank-one (\LSR) matrix 
subject to a two-norm constraint, i.e., for a given $x_k$,
\begin{equation} \label{eqn-trustProblem}
  \minimize{p\in\Re^n}\mgap\qDef(p) \defined g^Tp + \frac{1}{2} p^TB p
   \bgap\subject \mgap \|p\|_2 \le \delta,
\end{equation}
where $g\defined\nabla f(x_k)$, $B$ is an \LSR{} approximation to
$\nabla^2 f(x_k)$, and $\delta$ is a given positive constant.
In large-scale optimization, solving (\ref{eqn-trustProblem}) represents
the bulk of the computational effort in trust-region methods.
In this article, we propose a solver that is able to solve
(\ref{eqn-trustProblem}) to high accuracy.  

High-accuracy \LSR{} subproblem solvers are of interest in large-scale
optimization for two reasons: (1) In previous works, it has been shown that
more accurate subproblem solvers can require fewer overall trust-region
iterations, and thus, fewer overall function and gradient
evaluations~\cite{ipSSM,phasedSSM,ErwayM14}; and (2) it has been shown that
under certain conditions \SR{} matrices converge to the true Hessian--a
property that has not been proven for other quasi-Newton
updates~\cite{ConnGT91}.  While these convergence results have been proven
for \SR{} matrices, we are not aware of similar results
for \LSR{} matrices.  However, we hope that this paper will facilitate 
the study of \LSR{} quasi-Newton trust-region methods.


Solving large trust-region subproblems defined by indefinite matrices are
especially challenging, with optimal solutions lying on the boundary of the
trust-region.  Since \LSR{} matrices are not guaranteed to be positive
definite, additional care must be taken to handle indefiniteness and the
so-called \emph{hard case} (see, e.g., ~\cite{ConGT00a,MorS83}).  To our
knowledge, there are only two solvers designed to solve the quasi-Newton
subproblems to high accuracy for large-scale optimization.  Specifically,
the \MSS{} method~\cite{ErwayM14} is an adaptation of the Mor\'{e}-Sorensen
method~\cite{MorS83} to the limited-memory Broyden-Fletcher-Goldfarb-Shanno
(\LBFGS) quasi-Newton setting.  More recently, in~\cite{Burdakov15},
Burdakov et al. solve a trust-region subproblem where the trust region is
defined using \emph{shape-changing} norms.  It should be noted that while
the focus of~\cite{Burdakov15} is solving trust-region subproblems defined
by shape-changing norms instead of the usual Euclidean two-norm, Burdakov
et al. also present a trust-region method that is able to solve \LBFGS{}
quasi-Newton subproblems to high accuracy defined by the usual Euclidean
two-norm.  In this article, we present a method that extends what is
presented in~\cite{Burdakov15} to the indefinite case by handling three
additional non-trivial cases: (1) the singular case, (2) the so-called
\emph{hard case}, and (3) the general indefinite case.  We know of no
high-accuracy solvers designed specifically for \LSR{} trust-region
subproblems for large-scale optimization of the form
(\ref{eqn-trustProblem}) that are able to handle these cases associated
with \SR{} matrices.  It should be noted that large-scale solvers exist for
the general trust-region subproblem that are not designed to exploit
any specific structure of $B$.  Examples of these include
\LSTRS~\cite{rojas2001,Rojas2008} and \SSM~\cite{Hag01,HagP04}.

\medskip

Methods that solve the trust-region subproblem to high accuracy are
often based on the optimality conditions for a global solution
to the trust-region subproblem (see, e.g., Gay~\cite{Gay81},
Mor\'e and Sorensen \cite{MorS83} or Conn, Gould and
Toint \cite{ConGT00a}), given in the following theorem:

\medskip

\begin{theorem}\label{thrm-optimality}
  Let $\delta$ be a positive constant.  A vector $p^*$ is a global
  solution of the trust-region subproblem (\ref{eqn-trustProblem}) if and only
  if $\|p^*\|_2\leq \delta$ and there exists a unique $\sigma^*\ge 0$ such
  that $B+\sigma^* I$ is positive semidefinite and
\begin{equation}\label{eqn-optimality}
(B+\sigma^* I)p^*=-g \mgap \words{and} \mgap \sigma^*(\delta-\|p^*\|_2)=0.
\end{equation}
Moreover, if $B+\sigma^* I$ is positive definite, then the global
minimizer is unique.
\end{theorem}

The Mor\'{e}-Sorensen method~\cite{MorS83} seeks a solution pair of the
form $(p^*,\sigma^*)$ that satisfies both equations in
(\ref{eqn-optimality}) by alternating between updating $p^*$ and
$\sigma^*$; specifically, the method fixes $\sigma^*$, solving for $p^*$
using the first equation and then fixes $p^*$, solving for $\sigma^*$ using
the second equation.  In order to solve for $p^*$ in the first equation,
the Mor\'{e}-Sorensen method uses the Cholesky factorization of
  $B+\sigma I$; for this reason,
this method is prohibitively expensive for
  general large-scale optimization when $B$ does not have a 
structure that can be exploited.
However, the Mor\'{e}-Sorensen method is
arguably the best \emph{direct} method for solving the trust-region
subproblem.  While the Mor\'{e}-Sorensen direct method uses a
  safeguarded Newton method to find $\sigma^*$, the method proposed
in this article
  makes use of Newton method's together with a judicious initial
  guess so that safeguarding is not needed
  to obtain $\sigma^*$. Moreover, unlike the
Mor\'{e}-Sorensen method, the proposed method computes $p^*$ by formula,
and in this sense, is an \emph{iteration-free} method.


This article is organized in five sections.  In the second section, we
review the \LSR{} quasi-Newton matrices how to find its eigenvalues.  In
the third section, we describe the proposed \OBS{} method.  Numerical
results are presented in the fourth section, and concluding remarks are
found in the fifth section of the article.

\bigskip

\noindent  \textbf{Notation.}
Unless explicitly indicated, $\|\cdot\|$ denotes the vector two-norm or its
subordinate matrix norm.  The identity matrix is denoted by $I$, and its
dimension  depends on the context.  Finally, we assume that all \LSR{}
updates are computed so that the \LSR{} matrix is well
defined.

\section{L-SR1 matrices}\label{sec-LSR1}

In this section, we review \LSR{} matrices, their compact formulation, and how to compute
their eigenvalues.

\medskip

Given a continuously differentiable function $f(x)\in\Re^{n}$ and
iterates $\{x_k\}$, 
the \SR{} quasi-Newton method generates a sequence of matrices $\{B_k\}$
from a sequence of update pairs $\{(s_k,y_k)\}$ where
$$
	s_k\defined x_{k+1} - x_k 
	\quad 
	\text{and}
	\quad 
	y_k \defined
	\nabla f(x_{k+1}) - \nabla f(x_{k}),
$$ 
and $\nabla f$ denotes the gradient of $f$.  Given an initial matrix $B_0$,
$B_k$ is defined as
\begin{equation}\label{eqn-sr1}
	B_{k+1} \defined B_k + \frac{(y_k- B_ks_k)(y_k - B_ks_k)^T}{(y_k - B_ks_k)^Ts_k},
\end{equation}
provided $(y_k - B_ks_k)^Ts_k \ne 0$.  In practice, $B_0$ is often taken to
be a nonzero constant multiple of the identity matrix.  Limited-memory
\SR{} (\LSR) methods store and use only the $M$ most-recently computed
pairs $\{(s_k,y_k)\}$, where $M\ll n$.  Often $M$ may be very small (for
example, Byrd et al.~\cite{ByrNS94} suggest $M\in [3,7]$).
For more background on the \SR{} update
formula, please see, e.g., \cite{GrNaS09,KelS98,KhaBS93,NocW99,SunY06,Wol94}.

\bigskip



\noindent \textbf{Compact representation.} To compute a compact representation of an \SR{} matrix,
we make use of the following matrices:
\begin{eqnarray*}
	S_k &\defined& [ \ s_0 \ \ s_1 \ \ s_2 \ \ \cdots \ \ s_{k} \ ] \ \in \ \Re^{n \times (k+1)}, \\
	Y_k &\defined& [ \ y_0 \ \ y_1 \ \ y_2 \ \ \cdots \ \ y_{k} \ ] \ \in \ \Re^{n \times (k+1)}.
\end{eqnarray*}
Furthermore, we make use of the following decomposition of $S_k^TY_k \in\Re^{(k+1) \times (k+1)}$:
$$
	S_k^TY_k =   L_k + D_k + R_k,
$$
where $L_k$ is strictly lower triangular, $D_k$ is diagonal, and $R_k$ is
strictly upper triangular.  We assume all updates are well-defined, i.e.,
$(y_k - B_ks_k)^Ts_k \ne 0$; otherwise, the update is skipped (see \cite[Sec.\ 6.2]{NocW99}).

\medskip

The compact representation of \SR{} matrices is given by
Byrd et al.~\cite[Theorem 5.1]{ByrNS94}, who showed that $B_{k+1}$ in \eqref{eqn-sr1}
can be written in the form 
\begin{equation}\label{eqn-form}
	B_{k+1} \ = \ B_0 + \Psi_k M_k  \Psi_k^T,
\end{equation}
where $\Psi_k \in \Re^{n \times (k+1)}$, $M_k \in \Re^{(k+1) \times
  (k+1)}$, and $B_0$ is a diagonal matrix (i.e., $B_0=\gamma I$,
$\gamma\in\Re$).  In particular, $\Psi_k$ and $M_k$ are given by
\begin{equation*}
	\Psi_k \ = \ 
       	Y_k  - B_0S_k \quad \text{and} \quad
        M_k \ = \ (D_k + L_k + L_k^T - S_k^TB_0S_k)^{-1}.
\end{equation*}
Since $k\le M$, the matrix $M_k$ is small and can be inverted in practice.
We assume that updates are performed so that $\Psi_k$ always has full
column rank and that $\gamma \ne 0$ so that $B_0$ is invertible.

\bigskip


\noindent \textbf{Eigenvalues.}
  We now review from how to compute the eigenvalues of
  quasi-Newton matrices that admit a compact representation
  (\cite{Burdakov15,ErwaM15}).  For simplicity, we drop the subscript $k$, ($k\le
M\ll n$), and consider the problem of computing the eigenvalues of
\begin{equation}\label{eqn-Bk} 
	B = B_0 + \Psi M \Psi^T,
\end{equation}
where $B_0=\gamma I$, $\gamma\in\Re$. 
Suppose $\Psi = QR$ is the ``thin'' QR factorization of $\Psi$, where $Q
\in \Re^{n \times (k+1)}$ and $R \in \Re^{(k+1) \times (k+1)}$ is invertible.
Then,
\begin{equation}\label{eqn-eig-1}
	B = \gamma I + QRMR^TQ^T.
\end{equation}
Let $U\hat{\Lambda}U^T$ be the spectral
decomposition of $RMR^T \in \Re^{(k+1)\times(k+1)}$, where $U \in \Re^{(k+1) \times (k+1)}$ is
orthogonal and $\hat{\Lambda}=\diag(\hat{\lambda}_1, \dots,
\hat{\lambda}_{k+1})$, with $\hat{\lambda}_1 \le \dots \le
\hat{\lambda}_{k+1}$.  We note that since both $M$ and $R$ are invertible,
$\hat{\lambda}_i\ne 0$ for $i=1,\ldots, k+1$.  Substituting this into
(\ref{eqn-eig-1}), we obtain
\begin{equation*}
	B = \gamma I + QU\hat{\Lambda}U^TQ^T.
\end{equation*}
Since $Q$ and $U$ have orthonormal columns, then $P_\parallel \defined
QU\in\Re^{n\times (k+1)}$ also has orthonormal columns.  Let $P_\perp$ be a
matrix whose columns form an orthonormal basis for the orthogonal
complement of the column space of $P_\parallel$.  Then, $P \defined [ \
P_{\parallel} \ \ \ P_{\perp} ] \in \Re^{n \times n}$ is such that $P^TP = PP^T =
I$.  Thus, the spectral decomposition of $B$ is given by
%
%
\begin{equation}\label{eqn-Beig}
	B = P\Lambda P^T, \qquad 
	\text{where }
	\Lambda  \defined
	\begin{bmatrix}
		\Lambda_1 & 0 \\
		0 & \Lambda_2
	\end{bmatrix} = 
	\begin{bmatrix}
		\hat{\Lambda} + \gamma I & 0 \\
		0 & \gamma I
	\end{bmatrix},
\end{equation}
where $\Lambda=\diag(\lambda_1,\ldots,\lambda_n)
=\diag(\hat{\lambda}_1+\gamma, \ldots, \hat{\lambda}_{k+1}+\gamma, \gamma, \ldots, \gamma)$, 
$\Lambda_1
\in\Re^{(k+1)\times (k+1)}$, and
$\Lambda_2 
\in\Re^{(n-k-1)\times (n-k-1)}$.
 The remaining eigenvalues, found
on the diagonal of $\Lambda_2$, are equal to $\lambda_{k+2} = \gamma$.  
(For further details, see \cite{Burdakov15,ErwaM15}.)
For the duration of this article, we assume the first $k+1$ eigenvalues in
$\Lambda$ are ordered in increasing values, i.e.,
$\lambda_1\le
\lambda_2 \le \ldots \le \lambda_{k+1}$. Finally, throughout
this article we denote the leftmost eigenvalue of $B$ by $\lambda_{\min}$,
which is computed as $\lambda_{\min} = \min\{\lambda_1, \gamma\}$.

\section{Proposed method} \label{sec-proposed}

The method proposed in this paper, called the ``Orthonormal Basis
  \LSR{}'' (OBS) method, is able to solve the trust-region subproblem to
high accuracy even when the \LSR{} matrix is indefinite.  The method makes
use of two separate techniques.  One technique uses (1) a Newton method to
find $\sigma^*$ that is initialized so its iterates converge monotonically
to $\sigma^*$ without any safeguarding when global solutions lie on the
boundary of the trust region, and (2) the compact formulation of \SR{}
matrices together with the strategy found in ~\cite{BurWX96} to compute
$p^*$ directly by formula.  The other technique is newly proposed in this
article.  This technique computes an optimal pair $(p^*,\sigma^*)$ using an
orthonormal basis for the eigenspace of $B$.  The idea of using an
orthonormal basis to represent $p^*$ is not new; this approach is found
in~\cite{Burdakov15}.  In this manuscript, we apply this approach to the
cases when $B$ is singular and indefinite.

\medskip

We begin by providing an overview of the \OBS{} method.  To solve the
trust-region subproblem, we first attempt to compute an unconstrained
minimizer $p_u$ to \eqref{eqn-trustProblem}.  If the solution exists (i.e.,
$B$ is not singular) and lies inside the trust region, the optimal solution for
the trust-region subproblem is given by $p^*=p_u$ and $\sigma^*=0$.  This
computation is simplified by first finding the eigenvalues of $B$ 
(see \eqref{eqn-Beig}); the solution $p_u$ to the unconstrained
problem is found using a strategy found in~\cite{BurWX96}, adapted
for \LSR{} matrices.  If $\|p_u\| > \delta$ or is not well-defined,
a global solution of the trust-region subproblem must lie on the
boundary, i.e., it is a root of the following function:
\begin{equation} \label{eqn-phi}
	\phi(\sigma) = \frac{1}{\| p(\sigma) \|} - \frac{1}{\delta}.
\end{equation}

When a global solution is on the boundary, we consider three cases separately:
(i) $B$ is positive definite and $\| p_u \| > \delta$, 
(ii) $B$ is positive semidefinite, and
(iii) $B$ is indefinite.
We note that the so-called \textsl{hard case} can only occur in the third
case.  
Details for each case is provided below; however, we 
begin by considering the unconstrained case.
\\
\\
\noindent \textbf{Computing the unconstrained minimizer.}  The \OBS{}
method begins by computing the eigenvalues of $B$ as in
Section~\ref{sec-LSR1}.  If $B$ is positive definite, the \OBS{} method
computes $\|p_u\|$ using properties of orthogonal matrices.  If  
$\|p_u\|\le\delta$, 
then $(p^*,\sigma^*) = (p_u, 0)$.
We begin by presenting
the computation of $\|p_u\|$, which is only performed when $B$ is positive
definite.  We include $\sigma$ in the derivation for completeness even
though $\sigma=0$ when finding the unconstrained minimizer.

The unconstrained minimizer $p_u$ is the solution to the first optimality
condition in (\ref{eqn-optimality}); however, the unconstrained minimizer
can also be found by rewriting the optimality condition using the spectral
decomposition of $B$.  Specifically, suppose $B=P\Lambda P^T$ is the spectral
decomposition of $B$, then $$-g = (B+\sigma I)p = (P\Lambda P^T+\sigma
I)p = P(\Lambda+\sigma I)v,$$ where $v=P^Tp$.  Since $P$ is orthogonal,
$\|v\|=\|p\|$, and thus, the first optimality condition expressed in
(\ref{thrm-optimality}) can be written as
\begin{eqnarray} 
  (\Lambda + \sigma I)v &=& -P^Tg. \label{eqn-firstcond} 
\end{eqnarray}
Note that spectral decomposition of $B$
transforms the first system in (\ref{eqn-optimality})
into a solve with a diagonal matrix in (\ref{eqn-firstcond}).
If we express the right hand side as
$$
	P^Tg = [ \ P_{\parallel} \ \ \ P_{\perp} ]^T  g =  
	\begin{bmatrix}
		P_{\parallel}^Tg \\
		P_{\perp}^Tg
	\end{bmatrix}
	\defined
	\begin{bmatrix}
		g_{\parallel} \\
		g_{\perp}
	\end{bmatrix},
$$
then 
\begin{equation}\label{eqn-normv}
	\| p(\sigma) \|^2 = \| v(\sigma ) \|^2 
	= \left\{	\sum_{i = 1}^{k+1} \frac{ (g_{\parallel})_i^2 }{(\lambda_i + \sigma)^2}\right\} + \frac{ \| g_{\perp} \|^2}{(\gamma + \sigma)^2}.
\end{equation}
Thus, the length of the unconstrained minimizer $p_u = p(0)$ is computed as
$\| p_u \| = \|v(0) \|$, 
where $g_{\parallel} = P_{\parallel}^Tg = (QU)^Tg=(\Psi R^{-1} U)^Tg$ and $\| g_{\perp}\|^2 =
\|g\|^2 - \|g_\parallel\|^2$.  Notice that determining $\|p_u\|$ does not
require forming $p_u$ explicitly.  Moreover, 
we are able to compute $\| g_{\perp}\|$ without having to compute 
$g_\perp = P_{\perp}^Tg$, which requires computing $P_{\perp}$, whose columns form a basis orthogonal to $P_\parallel$.

If $\| p_u \| \le \delta$, then $p^* = p_u$ and $\sigma^* = 0$.  
To compute $p_u$, 
we use the Sherman-Morrison-Woodbury formula for the inverse of $B$
as in~\cite{BurWX96}, adapted from the \BFGS{} setting into the \SR{} setting in
\eqref{eqn-Bk}:
\begin{equation}\label{eqn-pstar}
p^* = -\frac{1}{\tau^*}\left[I-\Psi(\tau^* M^{-1}+\Psi^T\Psi)^{-1}\Psi^T)\right]g,
\end{equation}
where $\tau^* = \gamma$.  Notice that this formula calls for the inversion
of $(\tau^* M^{-1}+\Psi^T\Psi)$; however, the size of this matrix small
($(k+1)\times (k+1)$ where $k\le M$), making the computation practical.

On the other hand, if $\| p_u \| > \delta$, then the solution $p^*$
must lie on the boundary.  We now consider the three cases as mentioned
above.
 
\medskip

\noindent \textbf{Case (i): $B$ is positive definite and 
$\| p_u \| >
  \delta$.}  Since the unconstrained minimizer lies outside the trust
region and $\|p_u\|=\|p(0)\|$, then $\phi(\sigma)$ given by (\ref{eqn-phi}) is such that
$\phi(0) < 0$.  
%
In this case, the \OBS{} method
uses Newton's method to find $\sigma^*$.   (Details on
Newton's method are
provided in Section~\ref{sub-Newton}.)  Finally, setting $\tau^* = \gamma
+ \sigma^*$, the global solution of the trust-region subproblem, $p^*$, is
computed using \eqref{eqn-pstar}.

\medskip

\noindent

\noindent \textbf{Case (ii): $B$ is singular and positive semidefinite.}
Since $\gamma\ne 0$ and $B$ is positive semidefinite, 
the leftmost eigenvalue is $\lambda_1=0$.
Let $r$ be the multiplicity of the zero eigenvalue; that is,
$\lambda_1 = \lambda_2 = \ldots = \lambda_r = 0 < \lambda_{r+1}$.
For $\sigma > 0$, the matrix $(\Lambda + \sigma I)$ is invertible, and
thus, $\| p(\sigma) \|$ in \eqref{eqn-normv} is well-defined for
$\sigma\in(0, \infty)$.  If $\lim_{\sigma \rightarrow 0^+} \phi(\sigma) <
0$, 
the \OBS{} method uses Newton's method to find $\sigma^*$. (Details on
Newton's method are
provided in Section~\ref{sub-Newton}.)  Setting $\tau^*
= \gamma + \sigma^*$, $p^*$ is computed using \eqref{eqn-pstar}.

We now consider the remaining case: $\lim_{\sigma \rightarrow 0^+} \phi(\sigma)\ge 0$.  
By~\cite[Lemma
7.3.1]{ConGT00a}, $\phi(\sigma)$ is strictly increasing on the interval
$(0,\infty)$.  
Thus, $\phi$ can only have a root in the interval $[0,\infty]$ at $\sigma=0$.
We now show that $(p^*,\sigma^*)$ is a global solution of the
trust-region subproblem with $\sigma^*=0$ and
\begin{equation}\label{eqn-pseudoinv}
	p^* = -B^{\dagger}g = -P(\Lambda + \sigma^* I)^{\dagger}P^Tg,
\end{equation}
where ${\dagger}$ denotes the Moore-Penrose pseudoinverse.  The second
optimality condition holds in (\ref{eqn-optimality}) since $\sigma^*=0$.
It can be shown that the first optimality condition holds 
by using the fact that $g$ must be perpendicular to the eigenspace
corresponding to the 0 eigenvalue of $B$, i.e., 
$(P_{\parallel}^Tg)_i = 0$ for $i = 1, \ldots,r$ 
(see \cite{MorS83}).

In this subcase, the trust-region subproblem solution $p^*$ can
be computed as follows:
\begin{eqnarray}
	p^* 	
		&=& -P(\Lambda + \sigma^* I)^{\dagger}P^Tg \nonumber  \\
		&=&	\begin{cases}
				\displaystyle -P_{\parallel}(\Lambda_1+  \sigma^*I)^{\dagger} P_{\parallel}^Tg 
				- \frac{1}{\gamma + \sigma^*}P_{\perp}P_{\perp}^Tg 
				& \text{\phantom{\ $\Psi R^{-1}R^{-T}\Psi^T)g$}} \text{if $\sigma^* \ne -\gamma$,}\\
				-P_{\parallel}(\Lambda_1+  \sigma^*I)^{-1} P_{\parallel}^Tg 
				& \text{\phantom{\ $\Psi R^{-1}R^{-T}\Psi^T)g$}}  \text{otherwise}
			\end{cases}
			 \nonumber \\
		&=& \begin{cases}
			\displaystyle -\Psi R^{-1}U(\Lambda_1 +  \sigma^*I)^{\dagger} g_{\parallel}
			- \frac{1}{\gamma + \sigma^*} (I - \Psi R^{-1}R^{-T}\Psi^T)g & \text{if $\sigma^* \ne -\gamma$,}\\
			-\Psi R^{-1}U(\Lambda_1 +  \sigma^*I)^{-1} g_{\parallel} & \text{otherwise,}
			\end{cases}
			\label{eq:pseudoinv}
\end{eqnarray}
which makes use of the chain of following chain of equalities:
 $P_{\perp}P_{\perp}^Tg = (I - P_{\parallel}P_{\parallel}^T)g  =  (I - \Psi R^{-1}R^{-T}\Psi^T)g.$
The actual computation of $p^*$ in (\ref{eq:pseudoinv}) 
requires only matrix-vector
products; no additional
large matrices need to be formed to find a global solution of the
trust-region subproblem.


\medskip

\noindent \textbf{Case (iii): $B$ is indefinite.}  Since $B$ is indefinite,
$\lambda_{\min}=\min\{\lambda_1,\gamma\} < 0$.  Let $r$ be the algebraic
multiplicity of the leftmost eigenvalue.  For $\sigma > -\lambda_{\min}$,
$(\Lambda + \sigma I)$ is invertible, and thus, $\| p(\sigma) \|$ in
\eqref{eqn-normv} is well defined in the interval $(-\lambda_{\min},
\infty)$.  

If $\lim_{\sigma \rightarrow -\lambda_{\min}^+} \phi(\sigma) < 0$, then
there exists $\sigma^*\in (-\lambda_{\min}, \infty)$ with $\phi(\sigma^*) =
0$ that can obtained as in Case (i) using Newton's method (see Sec.~\ref{sub-Newton}).  
The solution $p^*$ is computed via \eqref{eqn-pstar} with $\tau^* = \gamma + \sigma^*$.

If $\lim_{\sigma \rightarrow -\lambda_{\min}^+} \phi(\sigma) \ge 0$, then $g$ must be 
orthogonal to the eigenspace associated with the leftmost eigenvalue of $B$ \cite{MorS83}.
In other words, if $\lambda_{\min}=\lambda_1$, then $(g_\parallel)_i=0$ for $i=1, \ldots, r$;
otherwise, if $\lambda_{\min}=\gamma$, then $\| g_{\perp} \| = 0$.
We now consider the cases of equality and inequality separately.

If
$\lim_{\sigma \rightarrow -\lambda_{\min}^+} \phi(\sigma) = 0$, then
$\sigma^* = -\lambda_{\min}> 0$, and a global solution of the
trust-region subproblem is given by
$$
	p^* = -(B+\sigma^* I)^{\dagger}g = -P(\Lambda + \sigma^* I)^{\dagger}P^Tg.
$$
As in Case (ii), $p^*$ is obtained from \eqref{eq:pseudoinv}
and can be shown to satisfy the optimality conditions (\ref{eqn-optimality}).

Finally, if
$\lim_{\sigma \rightarrow -\lambda_{\min}^+}\phi(\sigma) > 0$, then
$$
\lim_{\sigma \rightarrow -\lambda_{\min}^+} \| p(\sigma)\| =
\lim_{\sigma \rightarrow -\lambda_{\min}^+} \|
 -\left(B + \sigma^* I\right)^{-1}g\| <\delta.
 $$
This corresponds to the so-called \emph{hard case}.  
The optimal solution is given by 
\begin{equation}\label{eqn-pseudoinv-null}
	p^* = \hat{p}^* + z^*, \quad
\text{where} \quad \hat{p}^*=-\left(B + \sigma^* I\right)^{\dagger}g, \quad
z^*=\alpha u_{\min},
\end{equation}
and where $u_{\min}$ is an eigenvector associated with $\lambda_{\min}$ and $\alpha$
is computed so that $\|p^*\|=\delta$~\cite{MorS83}. 
As in Case
(ii), we avoid forming $P_\perp$ using (\ref{eq:pseudoinv}) to compute
$\hat{p}^*$.  
The computation of $u_{\min}$ depends
on whether $\lambda_{\min}$ is found in $\Lambda_1$ or $\Lambda_2$ in
(\ref{eqn-Beig}).  If $\lambda_{\min}=\lambda_1$ then
the first column of $P$ is a leftmost eigenvector
of $B$, and thus, $u_{\min}$ is set to the first column of $P_\parallel$.
On other hand, if $\lambda_{\min} = \gamma$, then any vector in the column
space of $P_{\perp}$ will be an eigenvector of $B$ corresponding to
$\lambda_{\min}$.  Since Range($P_{\parallel})^{\perp} = $
Range($P_{\perp}$), the projection matrix $(I -
P_{\parallel}P_{\parallel}^T)$ maps onto the column space of
$P_{\perp}$. For simplicity, we map one canonical basis vector at a time (starting with $e_1$)
into the space spanned by the columns of $P_\perp$ until we obtain
a nonzero vector.  Since $\dim(P_{\parallel})=k+1 \ll n$, this process
is practical and will result with a vector that lies in Range($P_\perp$);
that is, $u_{\min} \defined (I - P_{\parallel}P_{\parallel}^T)e_j$ for 
at least one $j$ in $\{1 \le j \le k+2\}$ with $\|u_{\min}\| \ne 0$. (We note
that both $\lambda_1$ and $\gamma$ cannot both be $\lambda_{\min}$
since $\lambda_1=\hat{\lambda}_1+\gamma$ and $\hat{\lambda}_1\ne 0$ (see
Section~\ref{sec-LSR1})).


\medskip

The following theorem provides details for computing
optimal trust-region subproblem solutions 
characterized by Theorem 1 for the case when $B$ is
indefinite. 

\medskip

\begin{theorem}\label{thrm-alpha}
  Consider the trust-region subproblem given by $$
  \minimize{p\in\Re^n}\mgap\qDef(p) \defined g^Tp + \frac{1}{2} p^TB p
  \bgap\subject \mgap \|p\|_2 \le \delta,$$ where $B$ is indefinite.
  Suppose $B=P\Lambda P^T$ is the spectral decomposition of $B$, and
  without loss of generality, assume
  $\Lambda=\diag(\lambda_1,\ldots,\lambda_n)$ is such that
  $\lambda_{\min}=\lambda_1\le \lambda_2 \le \ldots \le \lambda_n$.
  Further, suppose $g$ is orthogonal to the eigenspace associated with
  $\lambda_{\min}$, i.e., $g^TPe_j=0$ for $j=1,\ldots, r$, where $r\ge 1$
  is the algebraic multiplicity of $\lambda_{\min}$.  Then, if the optimal
  solution of the subproblem is with $\sigma^*=-\lambda_{\min}$, then the
  global solutions to the trust-region subproblem are given by $p^* =
  \hat{p}^*+z^*$ where $\hat{p}^*=-\left(B + \sigma^* I\right)^{\dagger}g$
  and $z^*=\pm \alpha u_{\min}$, where $u_{\min}$ is a unit vector in the
  eigenspace associated with $\lambda_{\min}$ and
  $\alpha=\sqrt{\delta^2-\|\hat{p}^*\|^2}$. Moreover,
\begin{equation}\label{eqn-roummel}
\qDef(\hat{p}^*\pm \alpha
  z^*)=\frac{1}{2}g^T\hat{p}^*-\frac{1}{2}\sigma^*\delta^2.
\end{equation}
\end{theorem}
\begin{proof}
  By~\cite{MorS83}, a global solution of trust-region
  subproblem is given by $p^* = \hat{p}^*+z^*$ where $\hat{p}^*=-\left(B +
    \sigma^*\right)^{\dagger}g$, $z^*= \bar{\alpha} \, u_{\min}$, and
  $\bar{\alpha}$ is such that $\|p^*\| = \delta$.  It remains to show that both
  roots of the quadratic equation $\|\hat{p}^*+\alpha
  u_{\min}\|^2=\delta^2$ are given by
  $\alpha=\pm\sqrt{\delta^2-\|\hat{p}^*\|^2}$ and that (\ref{eqn-roummel})
  holds.

To see this, we begin by showing that $\left(\hat{p}^*\right)^Tz^*=0$.  Let $r\ge 1$ be
the algebraic multiplicity of $\lambda_{\min}$.  Then,
$\hat{p}^*=-(B+\sigma^* I)^\dagger g = -P(\Lambda + \sigma^* I)^{\dagger}P^Tg= -P
v(\sigma^*)$, where $v(\sigma^*)\defined (\Lambda + \sigma^* I)^{\dagger}P^Tg$.
Notice that by definition of the pseudoinverse, $v(\sigma^*)_i=0$ for
$i=0,\ldots, r$. Since $u_{\min}$ is in the 
eigenspace associated with $\lambda_{\min}$, then it can be written
as a linear combination of the first $r$ columns of $P$, i.e.,
 $u_{\min}=\sum_{i=1}^{r} \tilde{u}_i Pe_i$ for some $\{\tilde{u}_i\}\in\Re$
where $e_i$ denotes the $i$th canonical basis vector.
Then,
 $$\left(\hat{p}^*\right)^Tz = \alpha
\left(\hat{p}^*\right)^Tu_{\min} = \alpha (Pv(\sigma^*))^T \left(  
 \sum_{i=1}^{r} \tilde{u}_i Pe_i \right) =
  \alpha v(\sigma^*)^T \sum_{i=1}^{r} \tilde{u}_i e_i = 0,$$ since the first $r$ entries of
  $v(\sigma^*)$ are zero.  Since $\hat{p}^*$ is orthogonal to $z^*$, then
  $$\alpha = \pm\sqrt{\delta^2-\|\hat{p}^*\|^2}.$$  

To see (\ref{eqn-roummel}), 
consider the following:
\begin{eqnarray} \label{eqn-roummel2}
\qDef( \hat{p}^* \pm \alpha \, u_{\min}) &= &  (\hat{p}^* \pm \alpha \, u_{\min})^Tg +
\frac{1}{2}(\hat{p}^* \pm \alpha \, u_{\min})^TB(\hat{p}^* \pm \alpha \, u_{\min}) \nonumber \\
& =&  (\hat{p}^* \pm \alpha \, u_{\min})^T(g-\frac{1}{2} g -\frac{1}{2} \sigma^*(\hat{p}^* \pm \alpha \, u_{\min}) )\nonumber \\
&=& \frac{1}{2}(\hat{p}^* \pm \alpha \, u_{\min})^Tg -\frac{1}{2} \sigma^*\| \hat{p}^* \pm \alpha \, u_{\min} \|^2 \nonumber \\
&=& \frac{1}{2}g^T\hat{p}^* -\frac{1}{2} \sigma^*\delta^2, 
\end{eqnarray}
since $u_{\min}^Tg =  \left(\sum_{i=1}^{r} \tilde{u}_i Pe_i\right)^Tg=
\left(\sum_{i=1}^{r} \tilde{u}_i e_i^TP^T\right)g =
0$ since $g$ is orthogonal to the eigenspace associated with $\lambda_{\min}$. $\square$
\end{proof}

\bigskip

The OBS method is summarized in Algorithm 1.

\begin{algorithm}[!h]
\SetAlgoNoLine
Compute $\Psi = QR$, the ``thin'' QR factorization (or, compute the Cholesky factor $R$ of $\Psi^T\Psi$);
\\
Compute $RMR^T = U\hat{\Lambda}U^T$ (the spectral decomposition)
with $\hat{\lambda}_1 \le \hat{\lambda}_2 \le \dots \le \hat{\lambda}_{k+1}$;  \\
Let $\Lambda_1 = \hat{\Lambda} + \gamma I$ (as in (\ref{eqn-Beig}));\\
Let $\lambda_{\min} = \min\{\lambda_1, \gamma\}$, and let $r$ be its algebraic multiplicity;\\
Define $P_{\parallel}  \defined \Psi R^{-1}U$ and $g_{\parallel} \defined P_{\parallel}^Tg$;  \\
Compute $a_j = (g_{\parallel})_j$ for $j = 1, \dots, k+1$ and $a_{k+2} = \sqrt{\| g \|_2^2 - \| g_{\parallel} \|_2^2}$;\\
\uIf{$\lambda_{\min} > 0$ \textnormal{ and } $\bar{\phi}(0) \ge 0$}{
		$\sigma^* = 0$ and compute $p^*$ from \eqref{eqn-pstar} with $\tau^*=\gamma$;
}
\uElseIf{$\lambda_{\min} \le 0$ \textnormal{ and } $\bar{\phi}(-\lambda_{\min}) \ge 0$}{
		$\sigma^* = -\lambda_{\min}$;\\
		Compute $p^*$ using \eqref{eq:pseudoinv};\\
		\If{$\lambda_{\min} < 0$}{
			Compute $z^*$ using \eqref{eqn-pseudoinv-null};\\
			$p^* \leftarrow p^* + 	z^*$;\\
		}
}
\Else{
		Use Newton's method to find $\sigma^*$, a root of $\phi$, in $(\max\{-\lambda_{\min}, 0 \}, \infty)$; \\
Compute $p^*$ from \eqref{eqn-pstar} with $\tau^*=\sigma^*+\gamma$;\\		
}
\caption{Orthonormal Basis SR1 method}\label{alg-SOBS}
\end{algorithm}


\subsection{Newton's method}\label{sub-Newton}
Newton's method is used to find a root of $\phi(\sigma)$ whenever
$$
\lim_{\sigma \rightarrow -\lambda_{\min}^+}\phi(\sigma) = \lim_{\sigma
  \rightarrow -\lambda_{\min}^+}\frac{1}{\| p(\sigma) \|} -
\frac{1}{\delta} < 0.
$$
Since $\| p(\sigma) \|$ does not exist at the eigenvalues of $B$, we first 
define the continuous extension of $\phi(\sigma)$,
whose domain is all of $\Re$.   Let $a_i = (g_{\parallel})_i$ for $1 \le
i \le k+1$, $a_{k+2} =  \| g_{\perp} \|$, and
$\lambda_{k+2} = \gamma$.  Combining the terms in \eqref{eqn-normv}
that correspond to the same eigenvalues and eliminating all terms with 
zero numerators, we have that  for $\sigma \ne \lambda_i$, $\| p(\sigma) \|^2$ can be written as
$$
	\| p(\sigma) \|^2 
	= \sum_{i=1}^{k+2} \frac{a_i^2}{(\lambda_i +\sigma)^2} 
	= \sum_{i=1}^{\ell} \frac{\bar{a}_i^2}{(\bar{\lambda}_i + \sigma)^2},
$$
such that for $i=1,\ldots, \ell$, 
$\bar{a}_i \ne 0$ and $\bar{\lambda}_i$ are \emph{distinct} eigenvalues of $B$
with $\bar{\lambda}_1 < \bar{\lambda}_2 < \dots < \bar{\lambda}_{\ell}$.
Note that the last sum is well-defined for $\sigma = \lambda_j \ne \bar{\lambda}_i$, for $1 \le i \le \ell$.
Then, the continuous extension $\bar{\phi}(\sigma)$ of $\phi(\sigma)$ is given by:
$$
	\bar{\phi}(\sigma) =
	\begin{cases}
		\displaystyle 
		-\frac{1}{\delta} & \text{if $\sigma = -\bar{\lambda}_i$, \,\, $1 \le i \le \ell$}\\
		\displaystyle \frac{1}{\sqrt{\displaystyle \sum_{i=1}^{\ell} \frac{{\bar{a}_i}^2}{(\bar{\lambda}_i+\sigma)^2}}}- \frac{1}{\delta} & \text{otherwise.}
	\end{cases}
$$
A crucial characteristic of $\bar{\phi}$ is that it
takes on the value of the 
limit of $\phi$ at $\sigma = -\lambda_i$, for $1 \le i \le k+2$.
In other words, for each
$i\in\{1,\ldots, k+2\}$,
$$
\lim_{\sigma\rightarrow -\lambda_i}\phi(\sigma) = 
\bar{\phi}(-\lambda_i). 
$$
The derivative of $\bar{\phi}(\sigma)$ is used only for Newton's method
and is computed as follows:
\begin{equation}\label{eqn-phiprime}
	\bar{\phi}'(\sigma) = 
	\left (
		\sum_{i=1}^{\ell}
		\frac{\bar{a}_i^2}{(\bar{\lambda}_i + \sigma)^2}
	\right )^{-\frac{3}{2}}
	\sum_{i=1}^{\ell}
	\frac{\bar{a}_i^2}{(\bar{\lambda}_i+\sigma)^3}
	\qquad \text{if $\sigma \ne -\bar{\lambda}_i, \,\, 1 \le i \le \ell$.}
\end{equation}
Note that $\bar{\phi}'(-\lambda_j)$ exists as long as 
$-\lambda_j \ne - \bar{\lambda}_i$, for $1 \le i \le \ell$.
Furthermore, for $\sigma > -\bar{\lambda}_1$,
$\bar{\phi}' (\sigma) > 0$, i.e., $\bar{\phi}(\sigma)$ is strictly increasing on the interval 
$[-\bar{\lambda}_1, \infty)$.
Finally, it can be shown that $\bar{\phi}''(\sigma) < 0$ for $\sigma > -\bar{\lambda}_1$,
i.e., $\bar{\phi}(\sigma)$ is concave on the interval $[-\bar{\lambda}_1, \infty)$.
For illustrative purposes, we plot examples of $\bar{\phi}(\sigma)$ in Fig.\ 1 for the different cases we considered in this Section~\ref{sec-proposed}.  Note that we use Newton's method to find $\sigma^*$ when
(a) $\lambda_{\min} \ge 0$ and $\bar{\phi}(0)$ (see Figs.\ 1(b) and (c)), or
(b) $\lambda_{\min} < 0$ and $\bar{\phi}(-\lambda_{\min}) < 0$ (see Figs.\ 1(d) and (e)).

\begin{figure}[htbp]\label{fig-phi}
	\begin{center}
	\begin{tabular}{cc}
	\includegraphics[width=.47\textwidth]{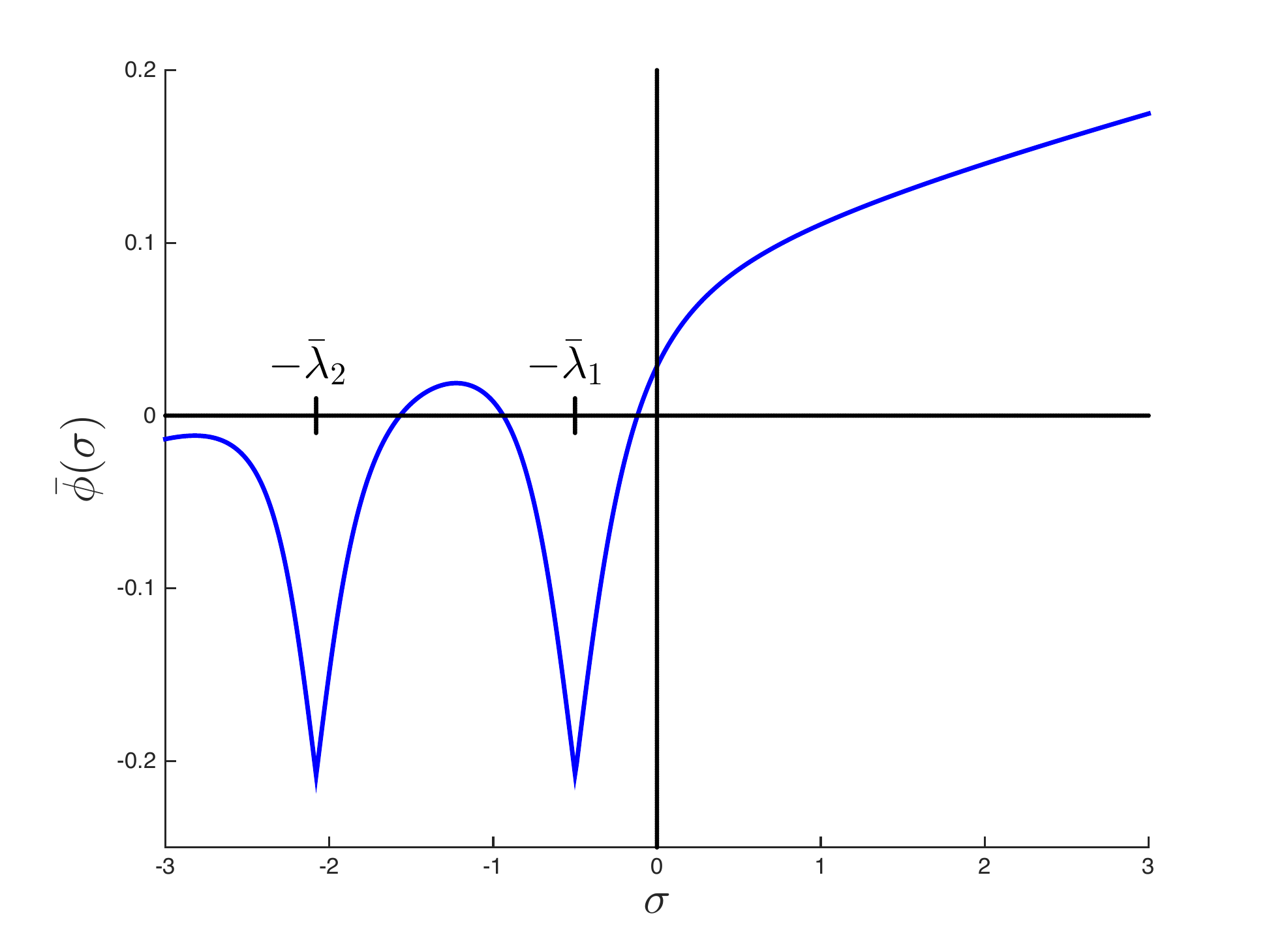} &
	\includegraphics[width=.47\textwidth]{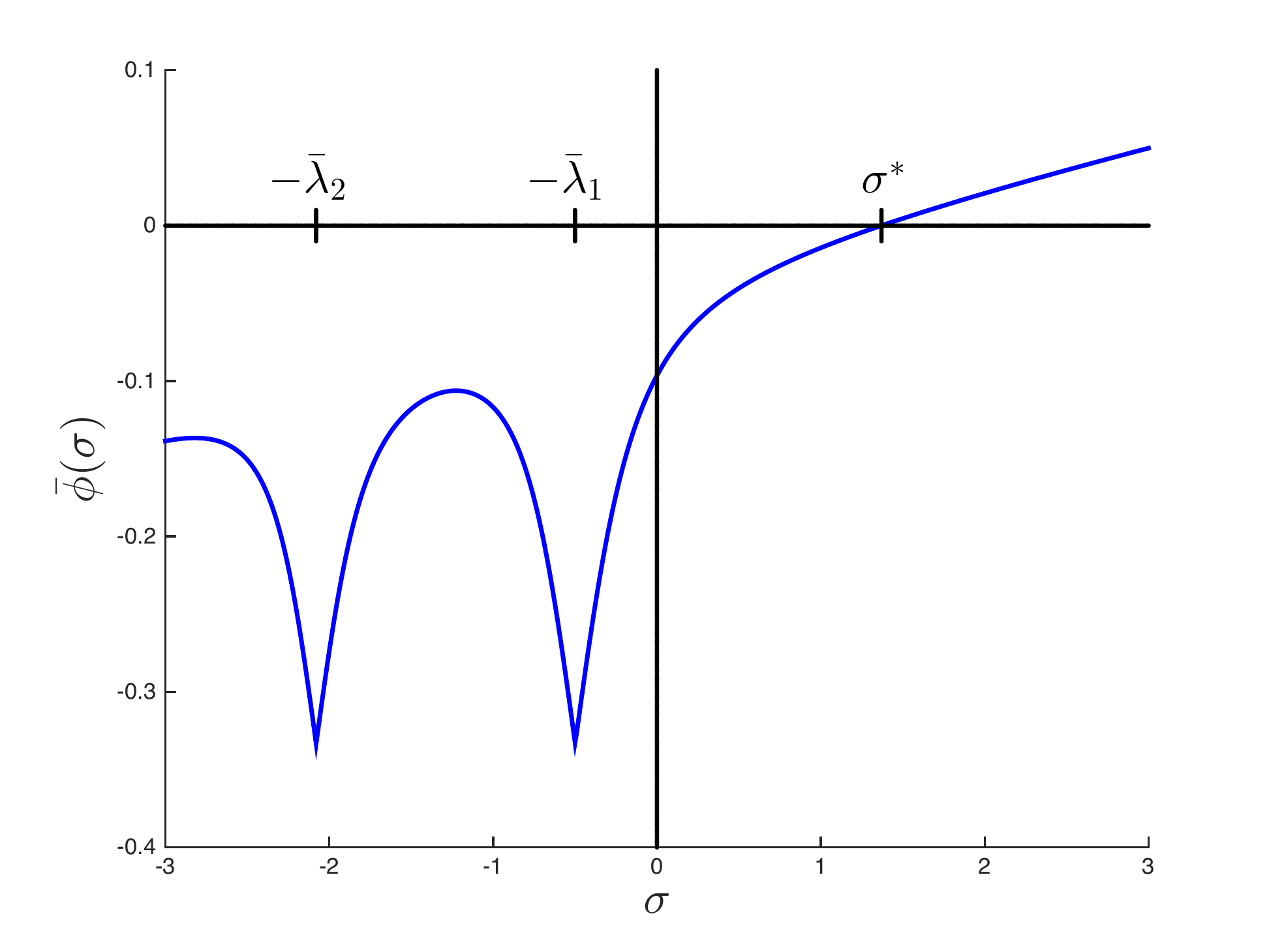} \\
	(a) & (b)\\
	\includegraphics[width=.47\textwidth]{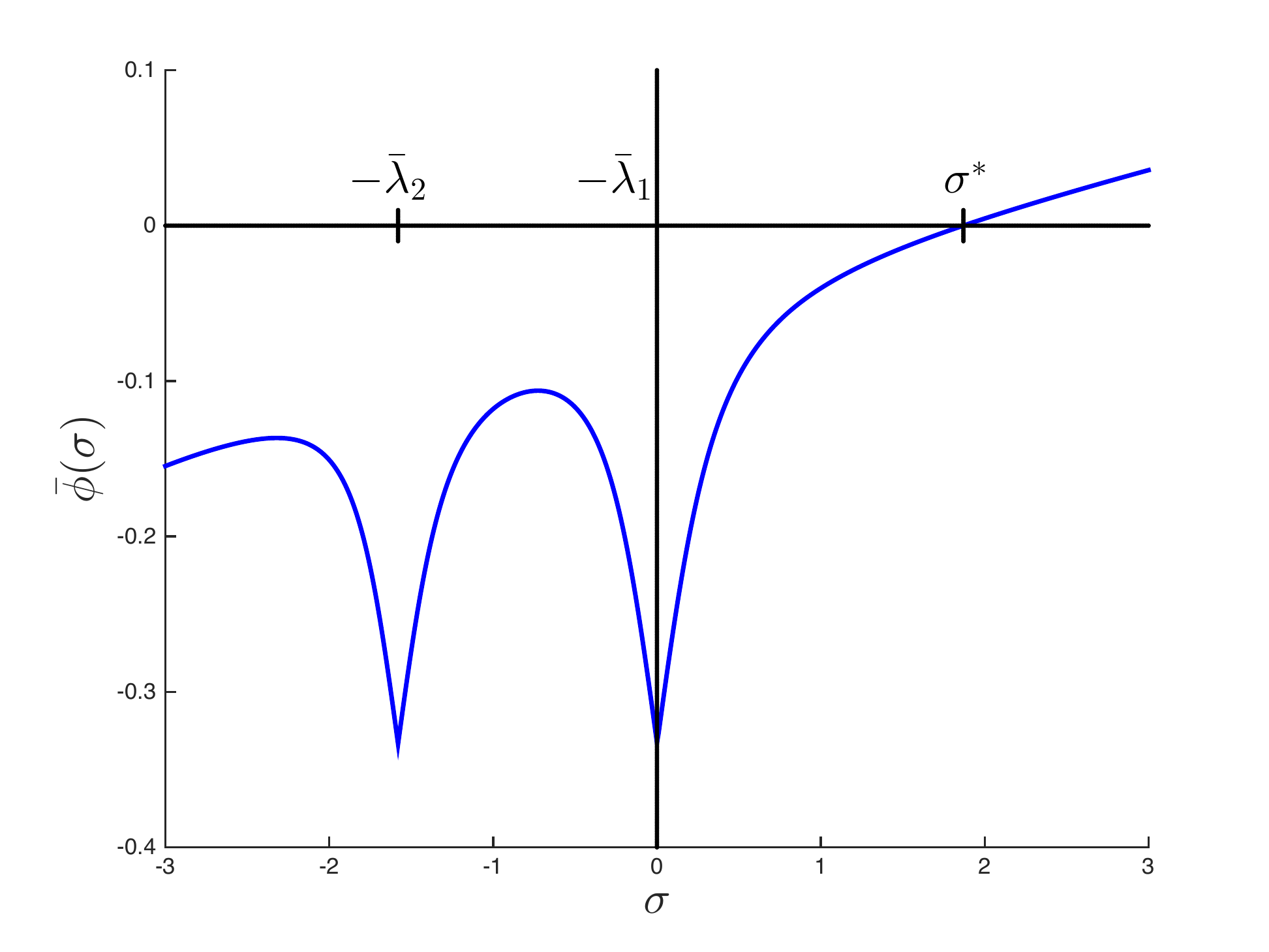} &
	\includegraphics[width=.47\textwidth]{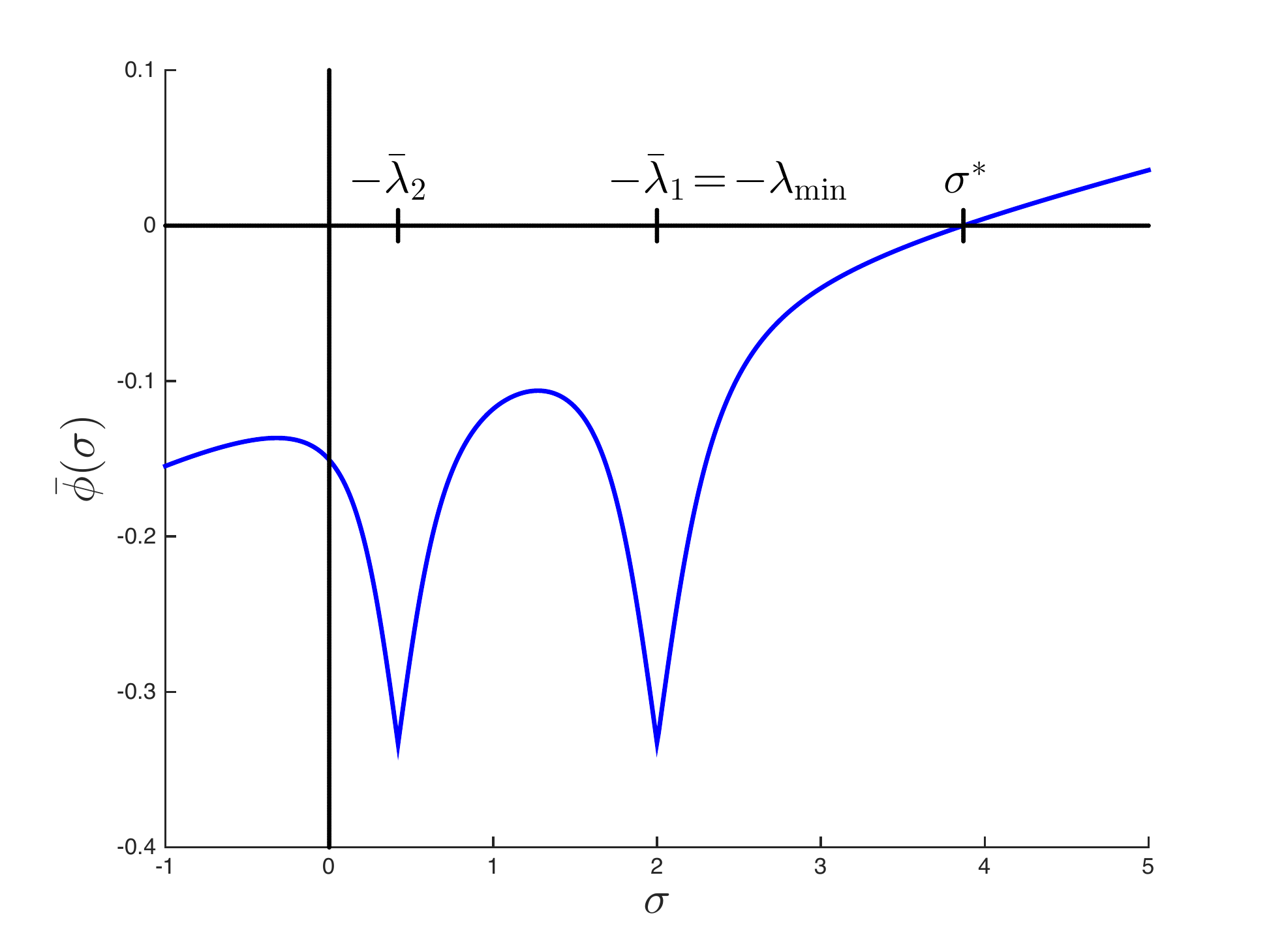} \\
	(c) & (d)\\
	\includegraphics[width=.47\textwidth]{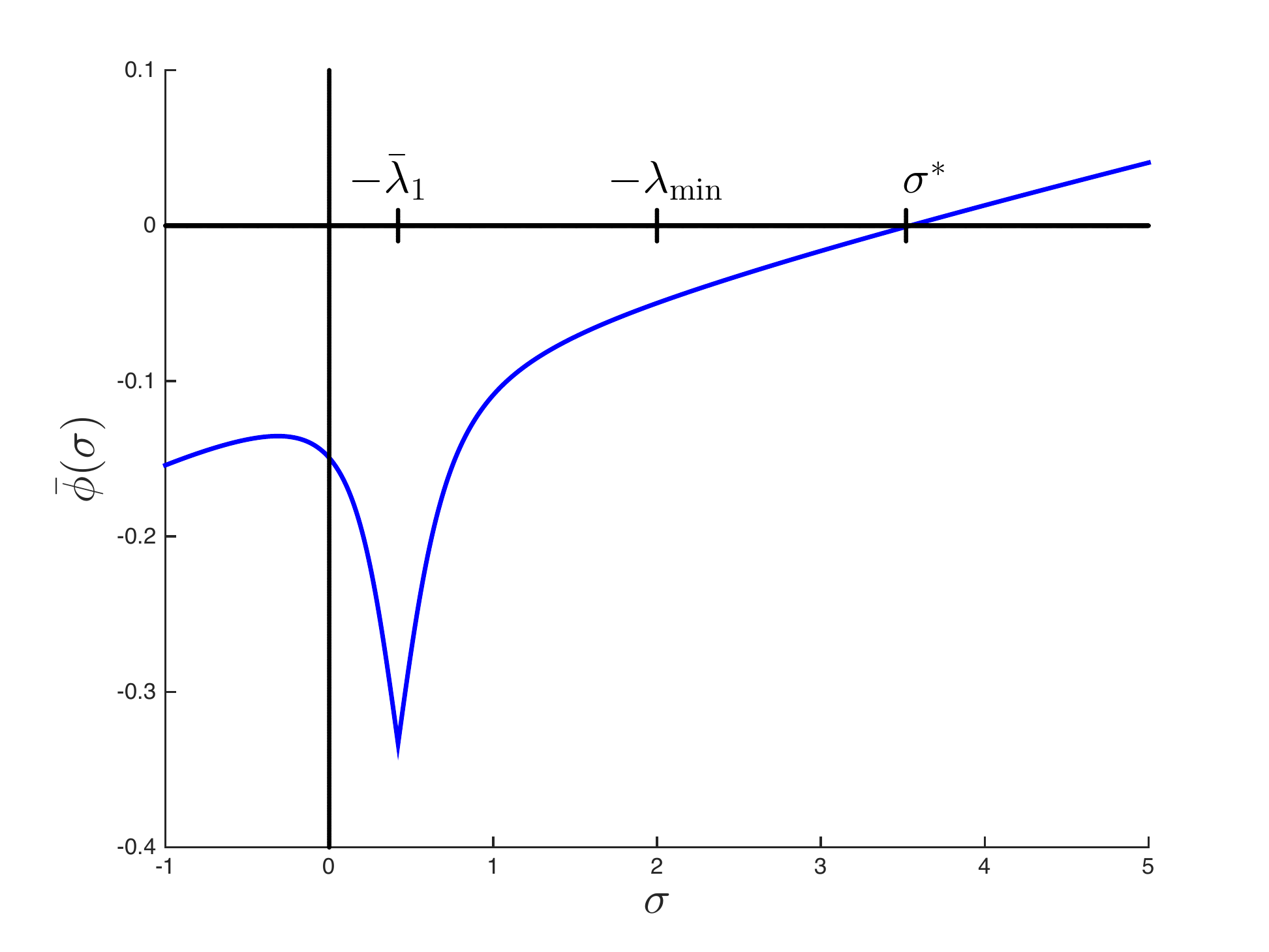} &
	\includegraphics[width=.47\textwidth]{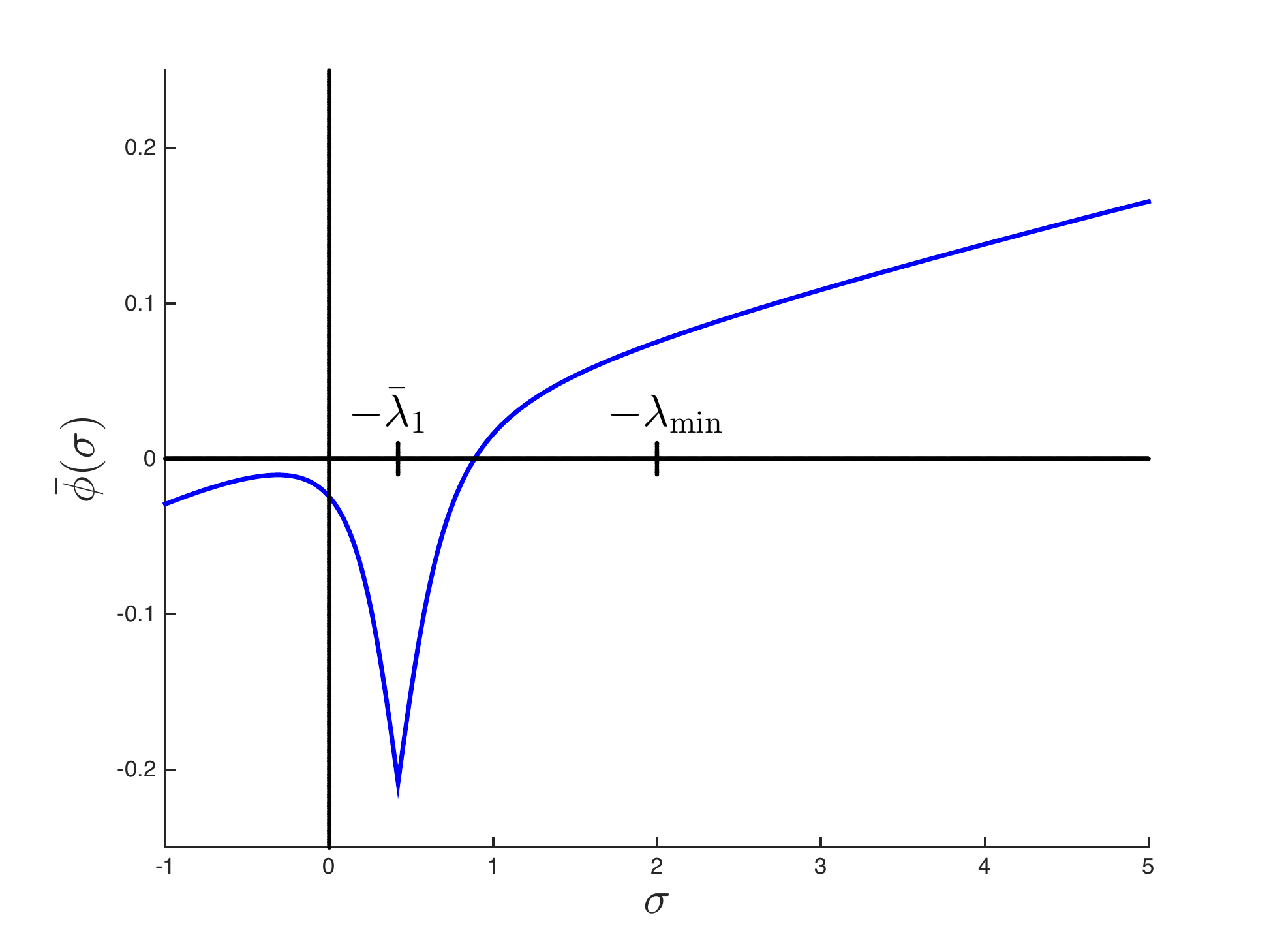} \\
	(e) & (f)\\
	\end{tabular}
	\end{center}
	\caption{Graphs of the function $\bar{\phi}(\sigma)$.
	(a) The positive-definite case where the unconstrained minimizer
	is within the trust-region radius, i.e., $\bar{\phi}(0) \ge 0$, and $\sigma^* = 0$.
	(b) The positive-definite case where the unconstrained minimizer is
	infeasible, i.e., $\bar{\phi}(0) < 0$.
	(c) The singular case where $\bar{\lambda}_1 = \lambda_{\text{min}} = 0$.
	(d) The indefinite case where $\bar{\lambda}_1 = \lambda_{\text{min}} < 0$.
	(e) When the coefficients $a_i$ corresponding to $\lambda_{\text{min}}$
	are all 0, $\bar{\phi}(\sigma)$ does not have a singularity at $\lambda_{\text{min}}$.
	Note that this case is not the hard case since $\bar{\phi}(-\lambda_{\min}) < 0$.
	(f) The hard case where there does not exist $\sigma^* > -\lambda_{\text{min}}$
	such that $\bar{\phi}(\sigma^*) = 0$.
	}
\end{figure}

We now define an initial iterate such that Newton's method is guaranteed to converge to $\sigma^*$ monotonically.

\begin{theorem}
Suppose 
$\bar{\phi}(\max\{ 0, -\lambda_{\min} \}) < 0$.
Let
\begin{equation}\label{eqn-hatsigma}
	\hat{\sigma} \defined   \max_{1 \le i \le k+2}\left\{ \frac{|a_i|}{\delta} - \lambda_i \right\} = 
	\frac{|a_j|}{\delta} - \lambda_j
      \end{equation}
for some $1 \le j \le k+2$.
Newton's method applied to $\bar{\phi}(\sigma)$ with initial iterate
$\sigma^{(0)} \defined \max \{ 0, \hat{\sigma} \}$
is guaranteed to converge to $\sigma^*$ monotonically.
\end{theorem}

\begin{proof} 
  Since $\bar{\phi}(\sigma)$ is strictly increasing and concave on
  $[-\lambda_{\min}, \infty)$ and $\bar{\phi}(\sigma^*) = 0$, it is
  sufficient to show that (i) $-\lambda_{\min} \le \sigma^{(0)} \le
  \sigma^*$, and (ii) $\bar{\phi}'(\sigma^{(0)})$ exists
  (see e.g., \cite{KinC02}). 
  
\medskip

We note that $\hat{\sigma} \ge -\lambda_{\min}$, and thus, $\sigma^{(0)}
\ge \max \{0, -\lambda_{\min}\} \ge -\lambda_{\min}$.  To show that
$\sigma^{(0)} \le \sigma^*$, we show that $\bar{\phi}(\sigma^{(0)}) \le
\bar{\phi}(\sigma^*) = 0$.

If $\hat{\sigma} = |a_j|/\delta - \lambda_j$
with $|a_j| \ne 0$, then evaluating $\|p(\sigma)\|$ at $\sigma=\hat\sigma$ yields
$$
		\| p(\hat{\sigma}) \|^2 \ = \  \sum_{i=1}^{k+2} \frac{a_i^2}{(\lambda_i + \hat{\sigma})^2} \\
				\ \ge \  \frac{a_j^2}{(\lambda_j + \hat{\sigma})^2}\\
				\ = \ \frac{a_j^2}{(\lambda_j + \frac{|a_j|}{\delta}- \lambda_j)^2}\\
				\ = \ \delta^2,
$$
and thus, $\bar{\phi}(\hat{\sigma})\le 0$.  
Since $\bar{\phi}(\max\{ 0, -\lambda_{\min} \}) < 0$, then $\bar{\phi}(\sigma^{(0)}) \le 0$.
If $|a_j| = 0$, then $\hat{\sigma} = -\lambda_j$.
Since $-\lambda_i \le -\lambda_{\min}$ for all $i$, $\hat{\sigma} = -\lambda_{\min}$.
Thus, $\bar{\phi}(\sigma^{(0)}) =\bar{\phi}(\max\{ 0, -\lambda_{\min} \}) < 0$.
Consequently, $\bar{\phi}(\sigma^{(0)}) \le 0$, and therefore, $\sigma^{(0)} \le \sigma^*$
since $\bar{\phi}(\sigma)$ is monotonically increasing.

Next, we show that $\bar{\phi}'(\sigma^{(0)})$ exists.  
On the interval $( -\lambda_{\min}, \infty)$, $\bar{\phi}(\sigma)$ is differentiable
(see \eqref{eqn-phiprime}). Therefore, if $\sigma^{(0)} > -\lambda_{\min}$, 
then $\bar{\phi}'(\sigma^{(0)})$ exists.  If $\sigma^{(0)} = -\lambda_{\min}$,
then $\hat{\sigma} = -\lambda_{\min}$, which implies that 
$a_1 = \cdots = a_r = 0$ or $a_{k+2}$ = 0 (see \eqref{eqn-hatsigma}).
From the definition of $\bar{\phi}(\sigma)$, $\lambda_{\min} \ne \bar{\lambda}_i$
for $1 \le i \le \ell$.  Thus, $\bar{\phi}(\sigma)$ is differentiable at 
$\sigma = -\lambda_{\min} = \sigma^{(0)}$.  $\square$

\end{proof}

We note that when $a_j \ne 0$ in \eqref{eqn-hatsigma}, $\hat{\sigma}$ is the largest $\sigma$
that solves the secular equation with the infinity norm:
$$
	\phi_{\infty}(\hat{\sigma}) = \frac{1}{\| v(\hat{\sigma}) \|_{\infty} } - \frac{1}{\delta} = 0.
$$
We illustrate the choice of initial iterate for Newton's method in Fig.\ 2.
\begin{figure}[htbp]\label{fig-phi}
	\begin{center}
	\begin{tabular}{cc}
	\includegraphics[width=.47\textwidth]{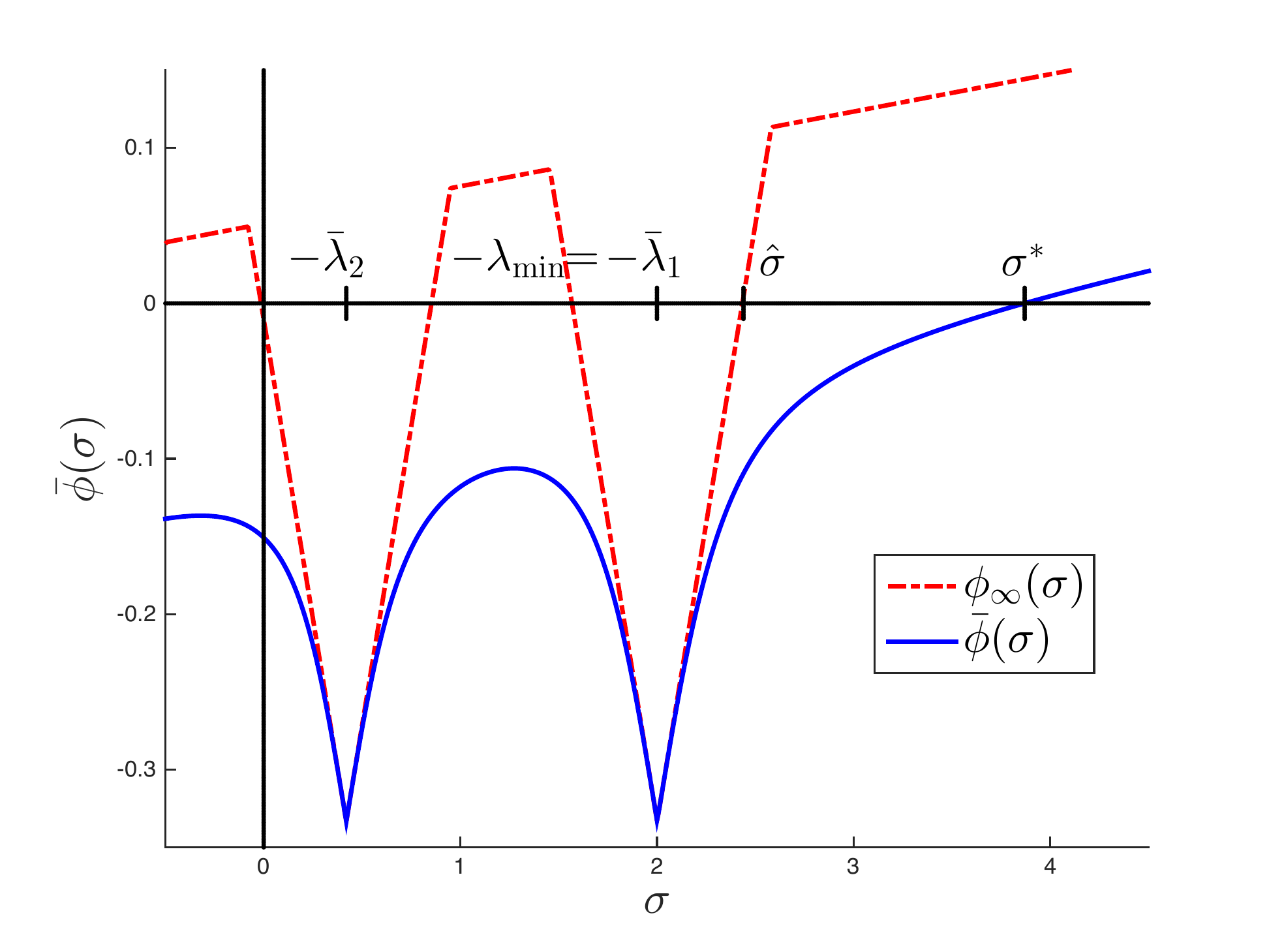} &
	\includegraphics[width=.47\textwidth]{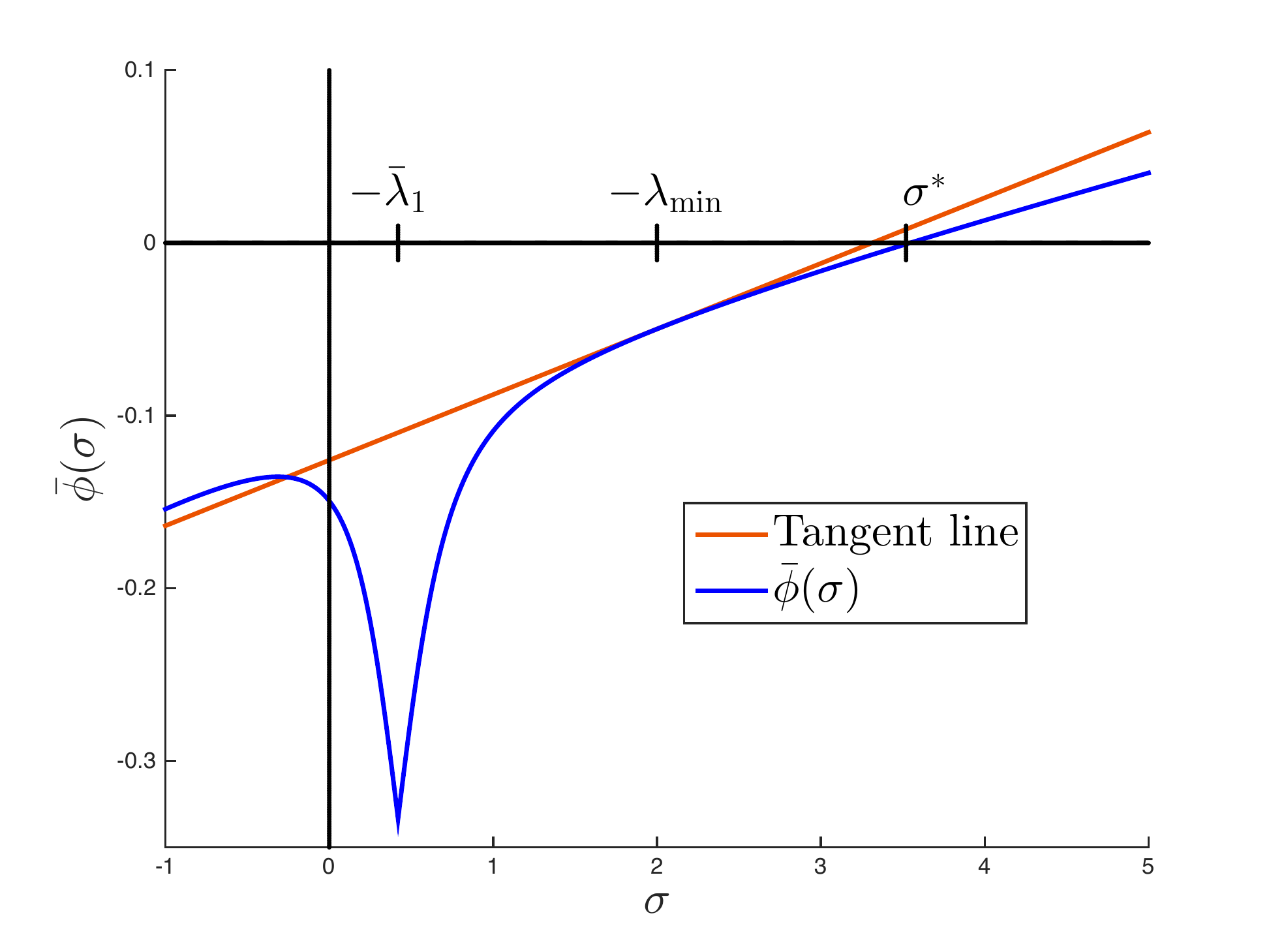} \\
	(a) & (b)\\
	\end{tabular}
	\end{center}
	\caption{Choice of initial iterate for Newton's method.
	(a) If $a_j \ne 0$ in \eqref{eqn-hatsigma}, then $\hat{\sigma}$ corresponds to the 
	largest root of $\phi_{\infty}(\sigma)$ (in red).
	Here, $-\lambda_{\min} > 0$, and therefore $\sigma^{(0)} = \hat{\sigma}$.
	(b) If $a_j = 0$ in \eqref{eqn-hatsigma}, then 
	$\lambda_{\min} \ne \bar{\lambda}_1$, and therefore, 
	$\bar{\phi}(\sigma)$ is differentiable at $-\lambda_{\min}$
	since $\bar{\phi}(\sigma)$ is differentiable on $(-\bar{\lambda}_1,\infty)$.
	Here, $-\lambda_{\min} > 0$, and thus, $\sigma^{(0)} = \hat{\sigma} = -\lambda_{\min}$.
	}
\end{figure}

Finally, we present Newton's method for computing $\sigma^*$.
\begin{algorithm}[!h]
\SetAlgoNoLine
Define tolerance $\tau > 0$;\\
\uIf{$\bar{\phi}(\max \{ 0, -\lambda_{\min}\}) < 0$}{
	$\hat{\sigma} =  \max_{1 \le j \le k+2} \frac{|a_j|}{\delta} - \lambda_j$;\\
	$\sigma = \max \{ 0, \hat{\sigma} \}$;\\
	\While{$|\bar{\phi}(\sigma)| > \tau$}{
		$\sigma = \sigma - \bar{\phi}(\sigma)/\bar{\phi}'(\sigma)$;
	}
	$\sigma^* = \sigma$;
}
\uElseIf{$\lambda_{\min} < 0$}{$\sigma^*=-\lambda_{\min}$;}
\Else{
	$\sigma^* = 0$;
}
\caption{Newton's method for computing $\sigma^*$}\label{alg-initpt}
\end{algorithm}

\medskip

\section{Numerical experiments}
In this section, we demonstrate the accuracy of the proposed \OBS{}
algorithm implemented in \MATLAB{} to solve limited-memory SR1 trust-region
subproblems.  For the experiments, five sets of experiments composed of
problems of various sizes were generated using random data.  The Newton
  method to find a root of $\phi$ was terminated when 
the $i$th iterate satisfied
  $\|\phi(\sigma^{(i)})\|\le \|\phi(\sigma^{(0)})\| + \tau$,
  where $\sigma^{(0)}$ denotes the initial iterate for Newton's method and $\tau
  = 1.0\times 10^{-10}$.  This is the only stopping criteria used by
the \OBS{} method since other
aspects of this method compute solutions by formula.  The problem sizes
$n$ range from $n=10^3$ to $n=10^7$.  The number of limited-memory updates
$k$ was set to 4, and thus $k+1=5$, and $\gamma=0.5$ unless otherwise
specified below. The pairs $S$ and $Y$, both $n\times (k+1)$ matrices, were
generated from random data.  Finally, $g$ was generated by random data
unless otherwise stated.  The five sets of experiments are intended to be
comprehensive: They include the unconstrained case and the three cases
discussed in Section~\ref{sec-proposed}.  The five experiments are as
follows:
\begin{enumerate}

\item The matrix $B$ is positive definite with $\| p_u \| \le \delta$: We ensure $\Psi$
  and $M$ are such that $B$ is strictly positive definite by altering the spectral decomposition of
  $RMR^T$.  We choose $\delta = \mu \| p_u \|$, where $\mu = 1.25$, to guarantee that the
  unconstrained minimizer is feasible.
  The graph of $\bar{\phi}(\sigma)$ corresponding to this case is illustrated in Fig.\ 1(a).

\item The matrix $B$ is positive definite with $\| p_u \| > \delta$: We ensure $\Psi$ and
  $M$ are such that $B$ is strictly positive definite by altering the spectral decomposition of
  $RMR^T$.  We choose $\delta = \mu \| p_u \|$, where $\mu$ is randomly 
  generated between 0 and 1, to guarantee that the unconstrained minimizer is infeasible.
  The graph of $\bar{\phi}(\sigma)$  corresponding to this case is illustrated in Fig.\ 1(b).

\item The matrix $B$ is positive semidefinite and singular with $p =
  -B^{\dagger}g$ infeasible: We ensure $\Psi$ and $M$ are such that $B$ is
  positive semidefinite and singular by altering the spectral decomposition
  of $RMR^T$.  Two cases are tested: (a) $\bar{\phi}(0)<0$ and (b)
  $\bar{\phi}(0)\ge 0$.  Case (a) occurs when $\delta = (1+\mu) \| p_u \|$,
  where $\mu$ is randomly generated between 0 and 1; case (b) occurs when
  $\delta = \mu \| p_u \|$, where $\mu$ is randomly generated between 0 and
  1.  The graph of $\bar{\phi}(\sigma)$ in case (a) corresponds to Fig.\
  1(c).  In case (b), $a_i = 0$ for $i = 1, \dots, r$, and thus,
    $\bar{\phi}(\sigma)$ does not have a singularity at $\sigma=0$, implying the graph
    of $\bar{\phi}(\sigma)$ for this case 
    corresponds to Fig \. 1(a).

\item The matrix $B$ is indefinite with 
  $\bar{\phi}(-\lambda_{\min}) < 0$: We ensure $\Psi$ and $M$ are
  such that $B$ is indefinite by altering the spectral decomposition of
  $RMR^T$. We test two subcases: (a) the vector $g$ is generated randomly,
  and (b) a random vector $g$ is projected onto the orthogonal complement of
  ${P_{\parallel}}_1 \in \Re^{n \times r}$ so that $a_i = 0, i=1,\ldots, r$,
  where $r=2$.  For case (b), $\delta=\mu\|p_u\|$, where
  $\mu$ is randomly generated between 0 and 1, so that 
  $\bar{\phi}(-\lambda_{\min}) < 0$.
  The graph of $\bar{\phi}(\sigma)$ in case (a) corresponds to Fig.\ 1(d), and $\bar{\phi}(\sigma)$
  in case (b) corresponds to Fig.\ 1(e).

 \item The hard case ($B$ is indefinite): We ensure $\Psi$ and $M$ are
  such that $B$ is indefinite by altering the spectral decomposition of
  $RMR^T$.  We test two subcases:
  (a) $ \lambda_{\min} =\lambda_1 =\hat{\lambda}_1+\gamma < 0 $,
  and (b) $ \lambda_{\min} = \gamma < 0 $.  In both cases, $\delta=(1+\mu)\|p_u\|$, where
  $\mu$ is randomly generated between 0 and 1, so that 
  $\bar{\phi}(-\lambda_{\min}) > 0$.
  The graph of $\bar{\phi}(\sigma)$ for both cases of the hard case corresponds to Fig.\ 1(f).

\end{enumerate}

We report the following: (1) \texttt{opt 1 (abs)}=$\| (B+\sigma^*I)p^* + g\|$, which corresponds
to the norm of the error in the first optimality conditions; 
(2) \texttt{opt 1 (rel)}
=$(\| (B+\sigma^*I)p^* + g\|)/\|g\|$, which corresponds
to the norm of the \emph{relative} error in the first optimality conditions;
(3) \texttt{opt 2}=$\sigma^*
|p^* - \delta |$, which corresponds to the absolute error in the second
optimality conditions; (4) $| \phi (\sigma^*) |$, which
measures how well the secular equation is satisfied; and (5) Time.  We ran
each experiment five times and report one representative result for each
experiment.  We show in Fig. 3 the computational time
for each of the five runs in each experiment.

\medskip

For comparison, we report results for the \OBS{} method as well as the
\LSTRS{} method~\cite{rojas2001,Rojas2008}.  The \LSTRS{} method solves
large trust-region subproblems by converting the subproblems into
parametrized eigenvalue problems.  This method uses only matrix-vector
products.  For these tests, we suppressed all run-time output of the
\LSTRS{} method and supplied a routine to compute matrix-vector products
using the factors in the compact formulation (see (\ref{eqn-Bk})), i.e.,
given a vector $v$, the product with $B$ is computed as $Bv\gets\gamma v +
\Psi(M(\Psi^Tv)).$ Note that the computations of $M$ and $\Psi$ are not
included in the time counts for \LSTRS.

\setlength\tabcolsep{1.5mm}
\begin{table}[!h]

  \caption{   Experiment 1: \OBS{} method with $B$ is
      positive definite and $\|p_u\|\le
      \delta$.} {
\centering
			\begin{tabular}{|c|c|c|c|c|c|}
                     \hline $n$   	& \texttt{opt 1 (abs)} &  \texttt{opt 1 (rel)} &  \texttt{opt 2} & $\sigma^* $ & Time  \\ 
\hline
 \texttt{1.0e+03} &  \texttt{3.24e-15}&  \texttt{ 1.03e-16} &     \texttt{0.00e+00}&    \texttt{0.00e+00} & \texttt{2.12e-02} \\ 
\hline
 \texttt{1.0e+04} &  \texttt{1.21e-14}&  \texttt{ 1.21e-16} &     \texttt{0.00e+00}&    \texttt{0.00e+00} & \texttt{2.76e-02} \\ 
\hline
 \texttt{1.0e+05} &  \texttt{4.61e-14}&  \texttt{ 1.46e-16} &     \texttt{0.00e+00}&    \texttt{0.00e+00} & \texttt{5.46e-02} \\ 
\hline
 \texttt{1.0e+06} &  \texttt{1.08e-13}&  \texttt{ 1.08e-16} &     \texttt{0.00e+00}&    \texttt{0.00e+00} & \texttt{5.34e-01} \\ 
\hline
 \texttt{1.0e+07} &  \texttt{5.31e-13}&  \texttt{ 1.68e-16} &     \texttt{0.00e+00}&    \texttt{0.00e+00} & \texttt{5.34e+00} \\ 
\hline
                      \end{tabular} }
\end{table}

\setlength\tabcolsep{1.5mm}
\begin{table}[!h]

  \caption{   Experiment 1: \LSTRS{} method with $B$ is
      positive definite and $\|p_u\|\le
      \delta$.} {
\centering
			\begin{tabular}{|c|c|c|c|c|c|}
                     \hline $n$   	& \texttt{opt 1 (abs)} &  \texttt{opt 1 (rel)} &  \texttt{opt 2} & $\sigma^* $ & Time  \\ 
\hline
 \texttt{1.0e+03} &  \texttt{2.11e-05}&  \texttt{ 6.70e-07} &     \texttt{0.00e+00}&    \texttt{0.00e+00} & \texttt{4.72e-01} \\ 
\hline
 \texttt{1.0e+04} &  \texttt{8.27e-07}&  \texttt{ 8.28e-09} &     \texttt{0.00e+00}&    \texttt{0.00e+00} & \texttt{4.98e-01} \\ 
\hline
 \texttt{1.0e+05} &  \texttt{2.64e-07}&  \texttt{ 8.37e-10} &     \texttt{0.00e+00}&    \texttt{0.00e+00} & \texttt{9.15e-01} \\ 
\hline
 \texttt{1.0e+06} &  \texttt{3.54e-09}&  \texttt{ 3.53e-12} &     \texttt{0.00e+00}&    \texttt{0.00e+00} & \texttt{7.08e+00} \\ 
\hline
 \texttt{1.0e+07} &  \texttt{2.79e-09}&  \texttt{ 8.81e-13} &     \texttt{0.00e+00}&    \texttt{0.00e+00} & \texttt{6.66e+01} \\ 
\hline
                      \end{tabular} }
\end{table}

Tables 1 and 2 shows the results of Experiment 1.  In all cases, the \OBS{}
method and the \LSTRS{} method found global solutions of the trust-region
subproblems.  The relative error in the \OBS{} method is smaller than the
relative error in the \LSTRS{} method.  Moreover, the \OBS{} method solved
each subproblem in less time than the \LSTRS{} method.

\begin{table}[!h]
  \caption{ Experiment 2: \OBS{} method with $ B$ is positive definite
    and $\|p_u\| > \delta$.} {
    \centering
    \begin{tabular}{|c|c|c|c|c|c|}
      \hline $n$   	& \texttt{opt 1 (abs)} &  \texttt{opt 1 (rel)} &  \texttt{opt 2} & $\sigma^* $ & Time  \\ 
      \hline
      \texttt{1.0e+03} &  \texttt{3.44e-15}&  \texttt{ 1.06e-16} &     \texttt{1.75e-09}&    \texttt{4.82e+01} & \texttt{2.83e-02} \\ 
      \hline \texttt{1.0e+04} &  \texttt{1.35e-14}&  \texttt{ 1.35e-16} &     \texttt{5.83e-13}&    \texttt{1.99e+01} & \texttt{2.70e-02} \\ 
      \hline \texttt{1.0e+05} &  \texttt{3.34e-14}&  \texttt{ 1.06e-16} &     \texttt{6.15e-13}&    \texttt{1.57e+01} & \texttt{6.39e-02} \\ 
      \hline \texttt{1.0e+06} &  \texttt{9.58e-14}&  \texttt{ 9.58e-17} &     \texttt{1.30e-11}&    \texttt{7.06e+01} & \texttt{5.38e-01} \\ 
      \hline \texttt{1.0e+07} &  \texttt{4.49e-13}&  \texttt{ 1.42e-16} &     \texttt{5.39e-06}&    \texttt{1.08e+00} & \texttt{5.37e+00} \\ 
      \hline
	\end{tabular} }
\end{table}

\begin{table}[!h]
  \caption{ Experiment 2: \LSTRS{} method with $ B$ is positive definite
    and $\|p_u\| > \delta$.} {
    \centering
    \begin{tabular}{|c|c|c|c|c|c|}
      \hline $n$   	& \texttt{opt 1 (abs)} &  \texttt{opt 1 (rel)} &  \texttt{opt 2} & $\sigma^* $ & Time  \\ 
\hline
 \texttt{1.0e+03} &  \texttt{1.32e-14}&  \texttt{ 4.05e-16} &     \texttt{6.25e-04}&    \texttt{4.82e+01} & \texttt{4.44e-01} \\ 
 \hline \texttt{1.0e+04} &  \texttt{1.20e-13}&  \texttt{ 1.20e-15} &     \texttt{1.20e-03}&    \texttt{1.99e+01} & \texttt{4.80e-01} \\ 
 \hline \texttt{1.0e+05} &  \texttt{5.45e-11}&  \texttt{ 1.73e-13} &     \texttt{4.90e-04}&    \texttt{1.57e+01} & \texttt{7.30e-01} \\ 
 \hline \texttt{1.0e+06} &  \texttt{4.68e-10}&  \texttt{ 4.68e-13} &     \texttt{1.35e-06}&    \texttt{7.06e+01} & \texttt{4.56e+00} \\ 
 \hline \texttt{1.0e+07} &  \texttt{4.15e-05}&  \texttt{ 1.31e-08} &     \texttt{4.47e-05}&    \texttt{1.08e+00} & \texttt{4.21e+01} \\ 
      \hline
	\end{tabular} }
\end{table}

Tables 3 and 4 show the results of Experiment 2.  In this case, the
unconstrained minimizer is not inside the trust region, making the value of
$\sigma^*$ strictly positive.  As in the first
experiment, the \OBS{} method appears to obtain solutions to higher
accuracy (columns 1, 2, and 3) and in less time (column 4) than the
\LSTRS{} method.  Finally, it is worth noting that as $n$ increases, the
accuracy of the solutions obtained by the \LSTRS{} method appears to
degrade.

\begin{table}[!h]
	\caption{ Experiment 3(a): \OBS{} method with $B$ is positive semidefinite and
singular with $\|B^\dagger g\| > \delta$.}  {
\centering
	\begin{tabular}{|c|c|c|c|c|c|}
      \hline $n$   	& \texttt{opt 1 (abs)} &  \texttt{opt 1 (rel)} &  \texttt{opt 2} & $\sigma^* $ & Time  \\ 
\hline
 \texttt{1.0e+03} &  \texttt{2.80e-14}&  \texttt{ 8.89e-16} &     \texttt{6.25e-10}&    \texttt{3.38e-01} & \texttt{2.70e-02} \\ 
\hline \texttt{1.0e+04} &  \texttt{1.17e-13}&  \texttt{ 1.16e-15} &     \texttt{1.18e-08}&    \texttt{1.03e-01} & \texttt{3.36e-02} \\ 
\hline \texttt{1.0e+05} &  \texttt{3.48e-12}&  \texttt{ 1.10e-14} &     \texttt{2.16e-07}&    \texttt{8.75e-03} & \texttt{6.43e-02} \\ 
\hline \texttt{1.0e+06} &  \texttt{1.44e-11}&  \texttt{ 1.44e-14} &     \texttt{1.48e-09}&    \texttt{3.62e-03} & \texttt{5.44e-01} \\ 
\hline \texttt{1.0e+07} &  \texttt{5.52e-10}&  \texttt{ 1.74e-13} &     \texttt{8.96e-09}&    \texttt{2.88e-03} & \texttt{5.39e+00} \\ 
					\hline
	\end{tabular} }
\end{table}
\begin{table}[!h]
	\caption{ Experiment 3(a): \LSTRS{} method with $B$ is positive semidefinite and
singular with $\|B^\dagger g\| > \delta$.}  {
\centering
	\begin{tabular}{|c|c|c|c|c|c|}
      \hline $n$   	& \texttt{opt 1 (abs)} &  \texttt{opt 1 (rel)} &  \texttt{opt 2} & $\sigma^* $ & Time  \\ 
 \hline\texttt{1.0e+03} &  \texttt{9.75e-03}&  \texttt{ 3.10e-04} &     \texttt{1.51e-16}&    \texttt{3.41e-01} & \texttt{4.78e-01} \\ 
 \hline \texttt{1.0e+04} &  \texttt{7.93e-02}&  \texttt{ 7.91e-04} &     \texttt{2.65e-15}&    \texttt{1.07e-01} & \texttt{5.69e-01} \\ 
 \hline \texttt{1.0e+05} &  \texttt{1.85e-01}&  \texttt{ 5.84e-04} &     \texttt{8.16e-16}&    \texttt{9.57e-03} & \texttt{1.56e+00} \\ 
 \hline \texttt{1.0e+06} &  \texttt{1.29e-01}&  \texttt{ 1.29e-04} &     \texttt{6.04e-16}&    \texttt{1.70e-03} & \texttt{1.28e+01} \\ 
 \hline \texttt{1.0e+07} &  \texttt{2.24e+03}&  \texttt{ 7.09e-01} &     \texttt{1.05e-10}&    \texttt{1.30e-06} & \texttt{6.39e+01} \\ 
					\hline
	\end{tabular} }
\end{table}

Tables 5 and 6 display the results of Experiment 3(a).  
This is experiment is the first of two in which $B$ is highly ill-conditioned.
In this experiment,
the \LSTRS{} method appears unable to obtain solutions to high absolute
accuracy (see column 2 in Table 6).  Moreover, the time required by the
\LSTRS{} to obtain solutions is, in some cases, significantly more than the
time required by the 
\OBS{} method.  In contrast, the \OBS{} method is able to obtain high
accuracy solutions.  Notice that the optimal values $\sigma^*$ found by
both methods appear to differ.  Global solutions to the subproblems solved
in Experiment 3(a) lie on the boundary of the trust region.  Because
\LSTRS{} was able to satisfy the second optimality condition to high
accuracy but not the first, this suggests \LSTRS's solution $p^*$ lies
on the boundary but there is some error in this solution.  As $n$
increases, the solution quality of the \LSTRS{} method appears to decline
with significant error in the case of $n=10^7$.  In this experiment, the
\OBS{} method appears to find solutions to high accuracy in comparable
time to other experiments; in contrast, the \LSTRS{} method appears to have
difficulty finding global solutions.

\begin{table}[!h]
	\caption{ Experiment 3(b): \OBS{} method with $B$ is positive semidefinite and
singular with $\|B^\dagger g\| \le \delta$.}  {
\centering
	\begin{tabular}{|c|c|c|c|c|c|}
      \hline $n$   	& \texttt{opt 1 (abs)} &  \texttt{opt 1 (rel)} &  \texttt{opt 2} & $\sigma^* $ & Time  \\ 
\hline
 \texttt{1.0e+03} &  \texttt{4.10e-15}&  \texttt{ 1.34e-16} &     \texttt{9.05e-10}&    \texttt{4.85e+01} & \texttt{3.01e-02} \\ 
\hline \texttt{1.0e+04} &  \texttt{1.01e-14}&  \texttt{ 1.02e-16} &     \texttt{1.34e-11}&    \texttt{6.98e+00} & \texttt{4.36e-02} \\ 
\hline \texttt{1.0e+05} &  \texttt{3.03e-14}&  \texttt{ 9.55e-17} &     \texttt{7.99e-14}&    \texttt{2.25e+01} & \texttt{6.70e-02} \\ 
\hline \texttt{1.0e+06} &  \texttt{1.39e-13}&  \texttt{ 1.39e-16} &     \texttt{4.18e-12}&    \texttt{3.42e+00} & \texttt{5.41e-01} \\ 
\hline \texttt{1.0e+07} &  \texttt{3.46e-13}&  \texttt{ 1.09e-16} &     \texttt{1.28e-11}&    \texttt{1.08e+00} & \texttt{5.37e+00} \\ 
					\hline
	\end{tabular} }
\end{table}

\begin{table}[!h]
	\caption{ Experiment 3(b): \LSTRS{} method with $B$ is positive semidefinite and
singular with $\|B^\dagger g\| \le \delta$.}  {
\centering
	\begin{tabular}{|c|c|c|c|c|c|}
      \hline $n$   	& \texttt{opt 1 (abs)} &  \texttt{opt 1 (rel)} &  \texttt{opt 2} & $\sigma^* $ & Time  \\ 
\hline \texttt{1.0e+03} &  \texttt{9.40e-15}&  \texttt{ 2.97e-16} &     \texttt{8.19e-04}&    \texttt{4.85e+01} & \texttt{4.42e-01} \\ 
\hline \texttt{1.0e+04} &  \texttt{2.06e-12}&  \texttt{ 2.07e-14} &     \texttt{6.59e-04}&    \texttt{6.98e+00} & \texttt{4.79e-01} \\ 
\hline \texttt{1.0e+05} &  \texttt{1.69e-11}&  \texttt{ 5.34e-14} &     \texttt{4.27e-05}&    \texttt{2.25e+01} & \texttt{7.43e-01} \\ 
\hline \texttt{1.0e+06} &  \texttt{6.27e-08}&  \texttt{ 6.28e-11} &     \texttt{6.19e-05}&    \texttt{3.42e+00} & \texttt{4.60e+00} \\ 
\hline \texttt{1.0e+07} &  \texttt{4.28e-05}&  \texttt{ 1.35e-08} &     \texttt{2.59e-05}&    \texttt{1.08e+00} & \texttt{6.29e+01} \\ 
					\hline
	\end{tabular} }
\end{table}

The results for Experiment 3(b) are shown in Tables 7 and 8.  This is the
second experiment involving ill-conditioned matrices.  As with Experiment
3(a), the \OBS{} method is able to obtain high-accuracy solutions in
generally less time than the \LSTRS{} method.  The accuracy obtained by the
\LSTRS{} method appears to degrade as the size of the problem increases.
In this experiment, the global solution always lies on the boundary, but the
larger residuals associated the second optimality condition in Table 8
indicate that the computed solutions by \LSTRS{} do not lie on the
boundary.


\begin{table}[!h]
  \caption{ Experiment 4(a): \OBS{} method with $ B $ is indefinite with 
  $\bar{\phi}(-\lambda_{\min}) < 0$.
  The vector $g$ is randomly generated. } {
  \centering
					\begin{tabular}{|c|c|c|c|c|c|}
      \hline $n$   	& \texttt{opt 1 (abs)} &  \texttt{opt 1 (rel)} &  \texttt{opt 2} & $\sigma^* $ & Time  \\ 
\hline
 \texttt{1.0e+03} &  \texttt{2.83e-15}&  \texttt{ 9.04e-17} &     \texttt{3.57e-12}&    \texttt{1.89e+02} & \texttt{3.05e-02} \\ 
\hline \texttt{1.0e+04} &  \texttt{1.27e-14}&  \texttt{ 1.27e-16} &     \texttt{1.53e-09}&    \texttt{1.18e+02} & \texttt{3.99e-02} \\ 
\hline \texttt{1.0e+05} &  \texttt{3.42e-14}&  \texttt{ 1.08e-16} &     \texttt{9.15e-13}&    \texttt{3.92e+02} & \texttt{6.40e-02} \\ 
\hline \texttt{1.0e+06} &  \texttt{1.19e-13}&  \texttt{ 1.20e-16} &     \texttt{4.79e-12}&    \texttt{5.39e+03} & \texttt{5.43e-01} \\ 
\hline \texttt{1.0e+07} &  \texttt{3.46e-13}&  \texttt{ 1.09e-16} &     \texttt{8.18e-11}&    \texttt{1.94e+04} & \texttt{5.35e+00} \\ 
					\hline
					\end{tabular}  }
\end{table}


\begin{table}[!h]
  \caption{ Experiment 4(a): \LSTRS{} method with $ B $ is indefinite with 
  $\bar{\phi}(-\lambda_{\min}) < 0$.
  The vector $g$ is randomly generated. } {
  \centering
					\begin{tabular}{|c|c|c|c|c|c|}
      \hline $n$   	& \texttt{opt 1 (abs)} &  \texttt{opt 1 (rel)} &  \texttt{opt 2} & $\sigma^* $ & Time  \\ 
\hline
 \texttt{1.0e+03} &  \texttt{4.92e-14}&  \texttt{ 1.57e-15} &     \texttt{5.40e-04}&    \texttt{1.89e+02} & \texttt{4.40e-01} \\ 
\hline \texttt{1.0e+04} &  \texttt{2.82e-14}&  \texttt{ 2.79e-16} &     \texttt{1.03e-03}&    \texttt{1.18e+02} & \texttt{4.80e-01} \\ 
\hline \texttt{1.0e+05} &  \texttt{2.11e-13}&  \texttt{ 6.69e-16} &     \texttt{2.68e-06}&    \texttt{3.92e+02} & \texttt{7.24e-01} \\ 
\hline \texttt{1.0e+06} &  \texttt{2.93e-11}&  \texttt{ 2.94e-14} &     \texttt{1.38e-07}&    \texttt{5.39e+03} & \texttt{4.49e+00} \\ 
\hline \texttt{1.0e+07} &  \texttt{1.81e-10}&  \texttt{ 5.74e-14} &     \texttt{3.19e-10}&    \texttt{1.94e+04} & \texttt{4.12e+01} \\ 
					\hline
					\end{tabular}  }
\end{table}

The results for Experiment 4(a) are displayed in Tables 9 and 10.  Both
methods found solutions that satisfied the first optimality conditions to
high accuracy.  The overall solution quality from the \OBS{} method appears
better in the sense that the residuals for both optimality conditions in
Table 9 are smaller than the residuals for both optimality conditions in
Table 10.  Finally, the \OBS{} method took less time to solve the
subproblem than the \LSTRS{} method.

\begin{table}[!h]
  \caption{ Experiment 4(b): \OBS{} method with $ B $ is indefinite with 
  $\bar{\phi}(-\lambda_{\min}) < 0$.
  The vector $g$ lies in the orthogonal complement of ${P_{\parallel}}_1$. } {
  \centering
					\begin{tabular}{|c|c|c|c|c|c|}
      \hline $n$   	& \texttt{opt 1 (abs)} &  \texttt{opt 1 (rel)} &  \texttt{opt 2} & $\sigma^* $ & Time  \\ 
\hline \texttt{1.0e+03} &  \texttt{3.42e-15}&  \texttt{ 1.07e-16} &     \texttt{1.17e-09}&    \texttt{1.31e+01} & \texttt{2.91e-02} \\ 
\hline \texttt{1.0e+04} &  \texttt{1.38e-14}&  \texttt{ 1.38e-16} &     \texttt{1.50e-14}&    \texttt{2.81e+00} & \texttt{3.16e-02} \\ 
\hline \texttt{1.0e+05} &  \texttt{3.17e-14}&  \texttt{ 1.00e-16} &     \texttt{3.55e-13}&    \texttt{1.82e+01} & \texttt{6.66e-02} \\ 
\hline \texttt{1.0e+06} &  \texttt{1.30e-13}&  \texttt{ 1.30e-16} &     \texttt{1.76e-12}&    \texttt{4.76e+00} & \texttt{5.46e-01} \\ 
\hline \texttt{1.0e+07} &  \texttt{3.14e-13}&  \texttt{ 9.94e-17} &     \texttt{4.36e-11}&    \texttt{7.58e+01} & \texttt{5.36e+00} \\ 
					\hline
					\end{tabular}  }
      \end{table}

\begin{table}[!h]
  \caption{ Experiment 4(b): \LSTRS{} method with $ B $ is indefinite with 
  $\bar{\phi}(-\lambda_{\min}) < 0$.
  The vector $g$ lies in the orthogonal complement of ${P_{\parallel}}_1$. } {
  \centering
					\begin{tabular}{|c|c|c|c|c|c|}
      \hline $n$   	& \texttt{opt 1 (abs)} &  \texttt{opt 1 (rel)} &  \texttt{opt 2} & $\sigma^* $ & Time  \\ 
\hline \texttt{1.0e+03} &  \texttt{1.16e-14}&  \texttt{ 3.64e-16} &     \texttt{1.24e-03}&    \texttt{1.31e+01} & \texttt{4.42e-01} \\ 
\hline \texttt{1.0e+04} &  \texttt{2.48e-12}&  \texttt{ 2.49e-14} &     \texttt{1.02e-04}&    \texttt{2.81e+00} & \texttt{4.70e-01} \\ 
\hline \texttt{1.0e+05} &  \texttt{1.50e-10}&  \texttt{ 4.75e-13} &     \texttt{2.82e-04}&    \texttt{1.82e+01} & \texttt{7.30e-01} \\ 
\hline \texttt{1.0e+06} &  \texttt{1.65e-08}&  \texttt{ 1.65e-11} &     \texttt{9.70e-05}&    \texttt{4.76e+00} & \texttt{4.65e+00} \\ 
\hline \texttt{1.0e+07} &  \texttt{2.08e-07}&  \texttt{ 6.58e-11} &     \texttt{1.06e-05}&    \texttt{7.58e+01} & \texttt{4.21e+01} \\ 
\hline
					\end{tabular}  }
      \end{table}

      The results of Experiment 4(b) are in Tables 11 and 12.  Both methods
      solved the subproblem to high accuracy as measured by the first
      optimality condition; however, the \OBS{} method solved the
      subproblem to significantly better accuracy as measured by the second
      optimality condition than the \LSTRS{} method.  All residual
      associated with the first and second optimality condition are less
      for the solution obtained by the \OBS{} method.  Moreover, the time
      required to find solutions was less for the \OBS{} method.

\begin{table}[!h]
  \caption{ Experiment 5(a): The \OBS{} method in the hard case ($B$ is indefinite) and $\lambda_{\min} =\lambda_1 =\hat{\lambda}_1+\gamma < 0$.}{
\centering
					\begin{tabular}{|c|c|c|c|c|c|}
      \hline $n$   	& \texttt{opt 1 (abs)} &  \texttt{opt 1 (rel)} &  \texttt{opt 2} & $\sigma^* $ & Time  \\ 
\hline \texttt{1.0e+03} &  \texttt{1.29e-14}&  \texttt{ 4.34e-16} &     \texttt{1.93e-16}&    \texttt{4.35e-01} & \texttt{3.38e-02} \\ 
\hline \texttt{1.0e+04} &  \texttt{5.87e-14}&  \texttt{ 5.86e-16} &     \texttt{2.59e-14}&    \texttt{6.08e-01} & \texttt{2.73e-02} \\ 
\hline \texttt{1.0e+05} &  \texttt{2.34e-12}&  \texttt{ 7.43e-15} &     \texttt{5.79e-14}&    \texttt{8.15e+00} & \texttt{8.08e-02} \\ 
\hline \texttt{1.0e+06} &  \texttt{1.33e-11}&  \texttt{ 1.33e-14} &     \texttt{1.19e-12}&    \texttt{3.97e+00} & \texttt{6.72e-01} \\ 
\hline \texttt{1.0e+07} &  \texttt{1.67e-10}&  \texttt{ 5.28e-14} &     \texttt{4.43e-12}&    \texttt{5.27e-01} & \texttt{6.71e+00} \\ 
					\hline
					\end{tabular}  } 
\end{table}

\begin{table}[!h]
  \caption{ Experiment 5(a): The \LSTRS{} method in the hard case ($B$ is indefinite) and $\lambda_{\min} =\lambda_1 =\hat{\lambda}_1+\gamma < 0$.}{
\centering
					\begin{tabular}{|c|c|c|c|c|c|}
      \hline $n$   	& \texttt{opt 1 (abs)} &  \texttt{opt 1 (rel)} &  \texttt{opt 2} & $\sigma^* $ & Time  \\ 
\hline \texttt{1.0e+03} &  \texttt{2.10e-05}&  \texttt{ 7.07e-07} &     \texttt{1.16e-15}&    \texttt{4.35e-01} & \texttt{4.70e-01} \\ 
\hline \texttt{1.0e+04} &  \texttt{3.88e+00}&  \texttt{ 3.87e-02} &     \texttt{1.50e-03}&    \texttt{6.08e-01} & \texttt{4.71e-01} \\ 
\hline \texttt{1.0e+05} &  \texttt{1.27e+02}&  \texttt{ 4.01e-01} &     \texttt{5.72e-04}&    \texttt{8.15e+00} & \texttt{7.65e-01} \\ 
\hline \texttt{1.0e+06} &  \texttt{2.04e+02}&  \texttt{ 2.04e-01} &     \texttt{1.45e-04}&    \texttt{3.97e+00} & \texttt{4.59e+00} \\ 
\hline \texttt{1.0e+07} &  \texttt{1.64e+03}&  \texttt{ 5.17e-01} &     \texttt{2.30e-05}&    \texttt{5.27e-01} & \texttt{4.23e+01} \\ 
					\hline
					\end{tabular}  } 
\end{table}

In the hard case with $\lambda_{\min}$ being a nontrivial eigenvalue, the
\OBS{} method was obtain global solutions to the subproblems; however, the
\LSTRS{} had difficulty finding high-accuracy solutions for all problem
sizes.  In particular, as $n$ increases, the solution quality of the
\LSTRS{} method appears to decline with significant error in the case of
$n=10^7$.  In all cases, the time required by the \OBS{} method to find a
solution was less than that of the time required by the \LSTRS{} method.

\begin{table}[!h]
\caption{ Experiment 5(b): The \OBS{} method in the hard case ($B$ is indefinite) and $ \lambda_{\min} = \gamma <0$. }   {
\centering
					\begin{tabular}{|c|c|c|c|c|c|}
      \hline $n$   	& \texttt{opt 1 (abs)} &  \texttt{opt 1 (rel)} &  \texttt{opt 2} & $\sigma^* $ & Time  \\ 
                                          \hline
\hline \texttt{1.0e+03} &  \texttt{3.52e-15}&  \texttt{ 1.11e-16} &     \texttt{3.53e-09}&    \texttt{6.35e+01} & \texttt{2.93e-02} \\ 
\hline \texttt{1.0e+04} &  \texttt{9.50e-15}&  \texttt{ 9.48e-17} &     \texttt{1.16e-14}&    \texttt{2.10e+02} & \texttt{3.82e-02} \\ 
\hline \texttt{1.0e+05} &  \texttt{3.01e-14}&  \texttt{ 9.50e-17} &     \texttt{4.49e-13}&    \texttt{4.49e+02} & \texttt{6.71e-02} \\ 
\hline \texttt{1.0e+06} &  \texttt{9.48e-14}&  \texttt{ 9.47e-17} &     \texttt{6.86e-12}&    \texttt{1.34e+04} & \texttt{5.32e-01} \\ 
\hline \texttt{1.0e+07} &  \texttt{3.40e-13}&  \texttt{ 1.07e-16} &     \texttt{2.97e-12}&    \texttt{8.91e+03} & \texttt{5.36e+00} \\ 
\hline					\end{tabular} }
\end{table}

\begin{table}[!h]
\caption{ Experiment 5(b): The \LSTRS{} method in the hard case ($B$ is indefinite) and $ \lambda_{\min} = \gamma <0$. }   {
\centering
					\begin{tabular}{|c|c|c|c|c|c|}
      \hline $n$   	& \texttt{opt 1 (abs)} &  \texttt{opt 1 (rel)} &  \texttt{opt 2} & $\sigma^* $ & Time  \\ 
                                          \hline
 \texttt{1.0e+03} &  \texttt{2.24e-14}&  \texttt{ 7.12e-16} &     \texttt{7.36e-04}&    \texttt{6.35e+01} & \texttt{4.41e-01} \\ 
\hline \texttt{1.0e+04} &  \texttt{6.35e-14}&  \texttt{ 6.33e-16} &     \texttt{1.92e-06}&    \texttt{2.10e+02} & \texttt{5.02e-01} \\ 
\hline \texttt{1.0e+05} &  \texttt{2.26e-13}&  \texttt{ 7.14e-16} &     \texttt{5.09e-08}&    \texttt{4.49e+02} & \texttt{7.49e-01} \\ 
\hline \texttt{1.0e+06} &  \texttt{6.61e-12}&  \texttt{ 6.61e-15} &     \texttt{4.76e-08}&    \texttt{1.34e+04} & \texttt{4.32e+00} \\ 
\hline \texttt{1.0e+07} &  \texttt{8.77e-11}&  \texttt{ 2.77e-14} &     \texttt{1.05e-08}&    \texttt{8.91e+03} & \texttt{4.09e+01} \\ 
\hline					\end{tabular} }
\end{table}

The results of Experiment 5(b) are in Tables 15 and 16.  Unlike in
Experiment 5(a), the \LSTRS{} method was able to find solutions to high
accuracy.  In all cases, the \OBS{} method was able to find solutions with
higher accuracy than the \LSTRS{} method and in less time.

\begin{figure}[!h]
\label{fig:time}
\centering
\begin{tabular}{cc}
\includegraphics[height=3.9cm]{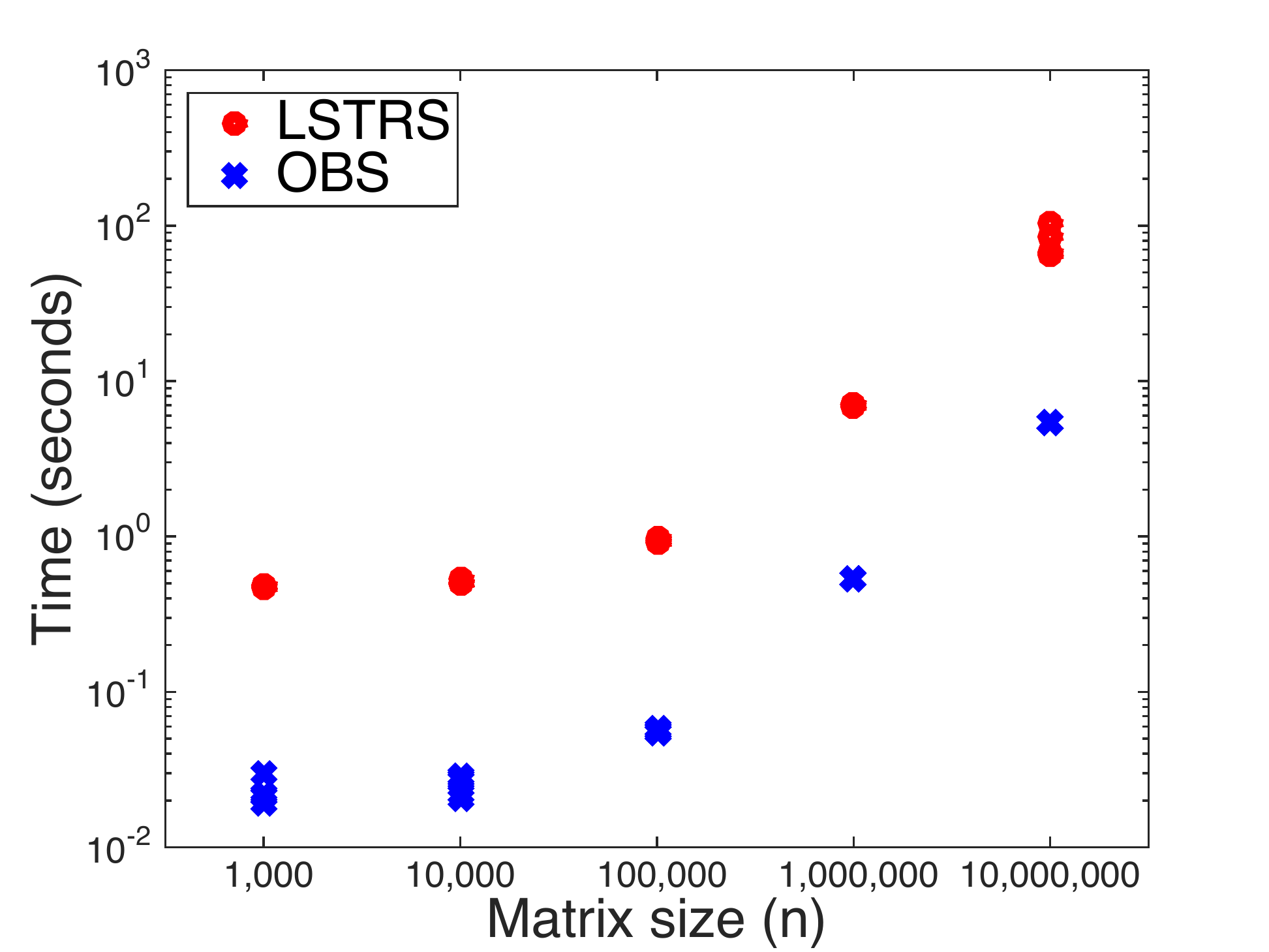} &
\includegraphics[height=3.9cm]{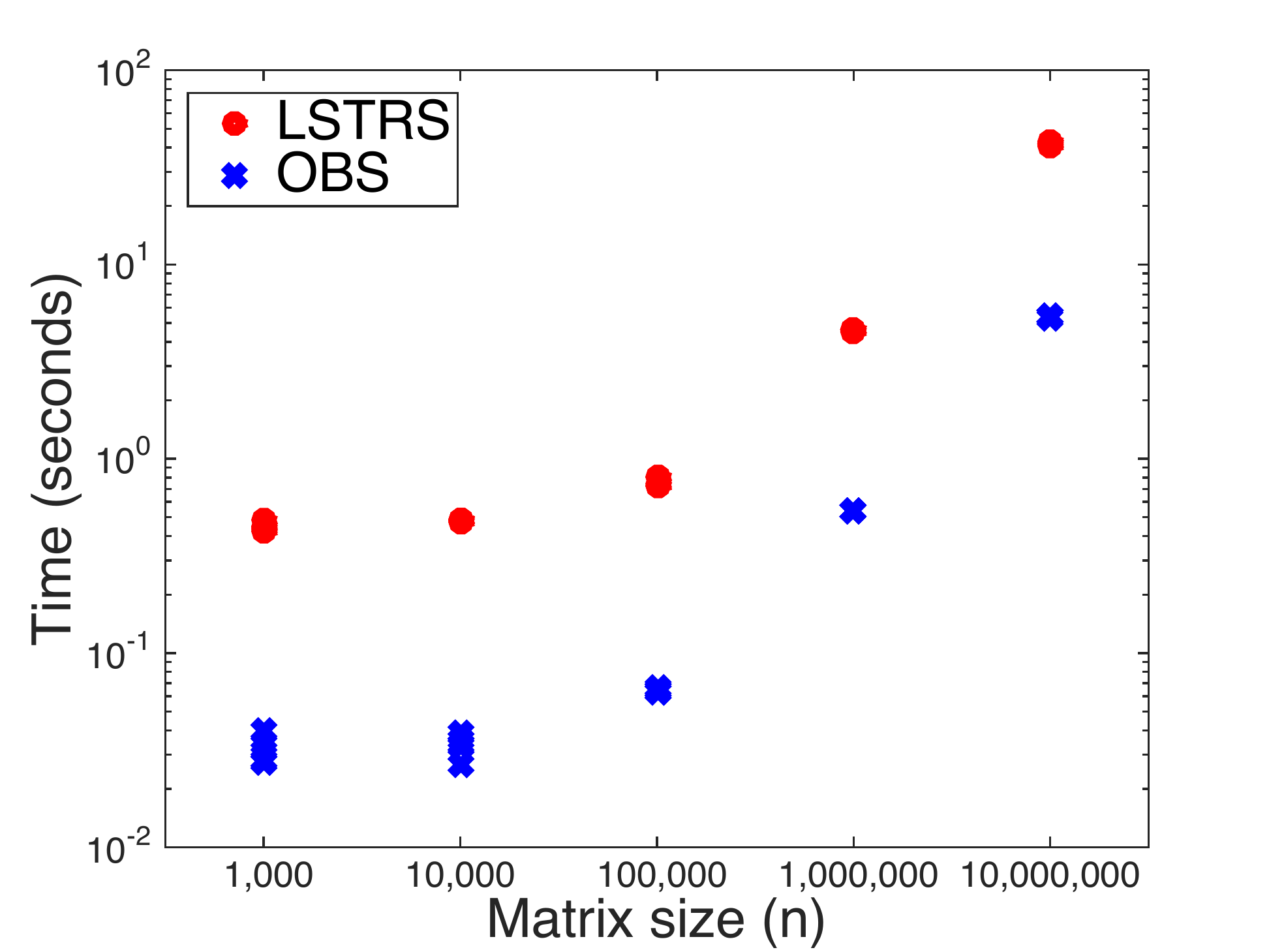} \\
Experiment 1 & Experiment 2 \\
\includegraphics[height=3.9cm]{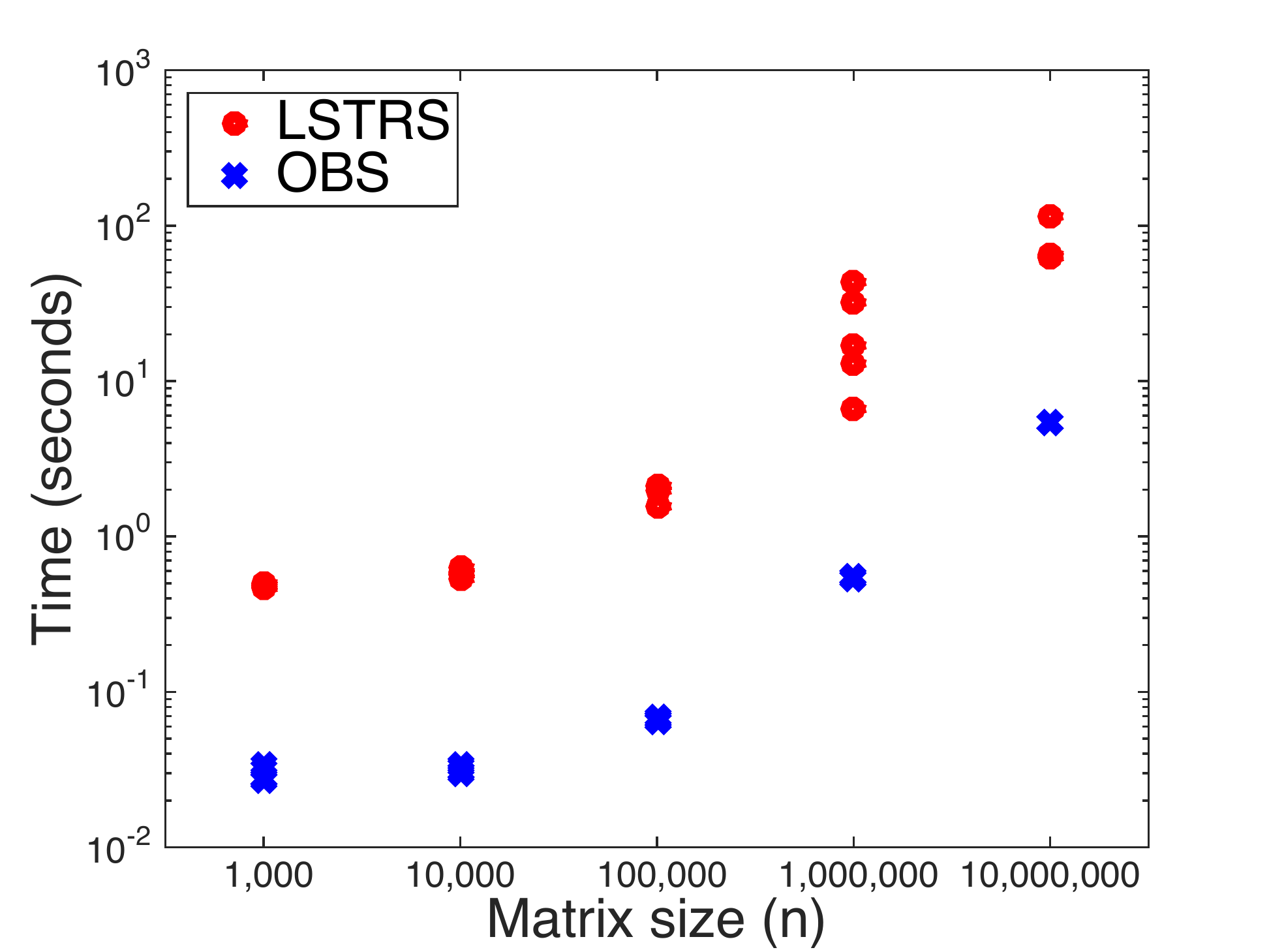} &
\includegraphics[height=3.9cm]{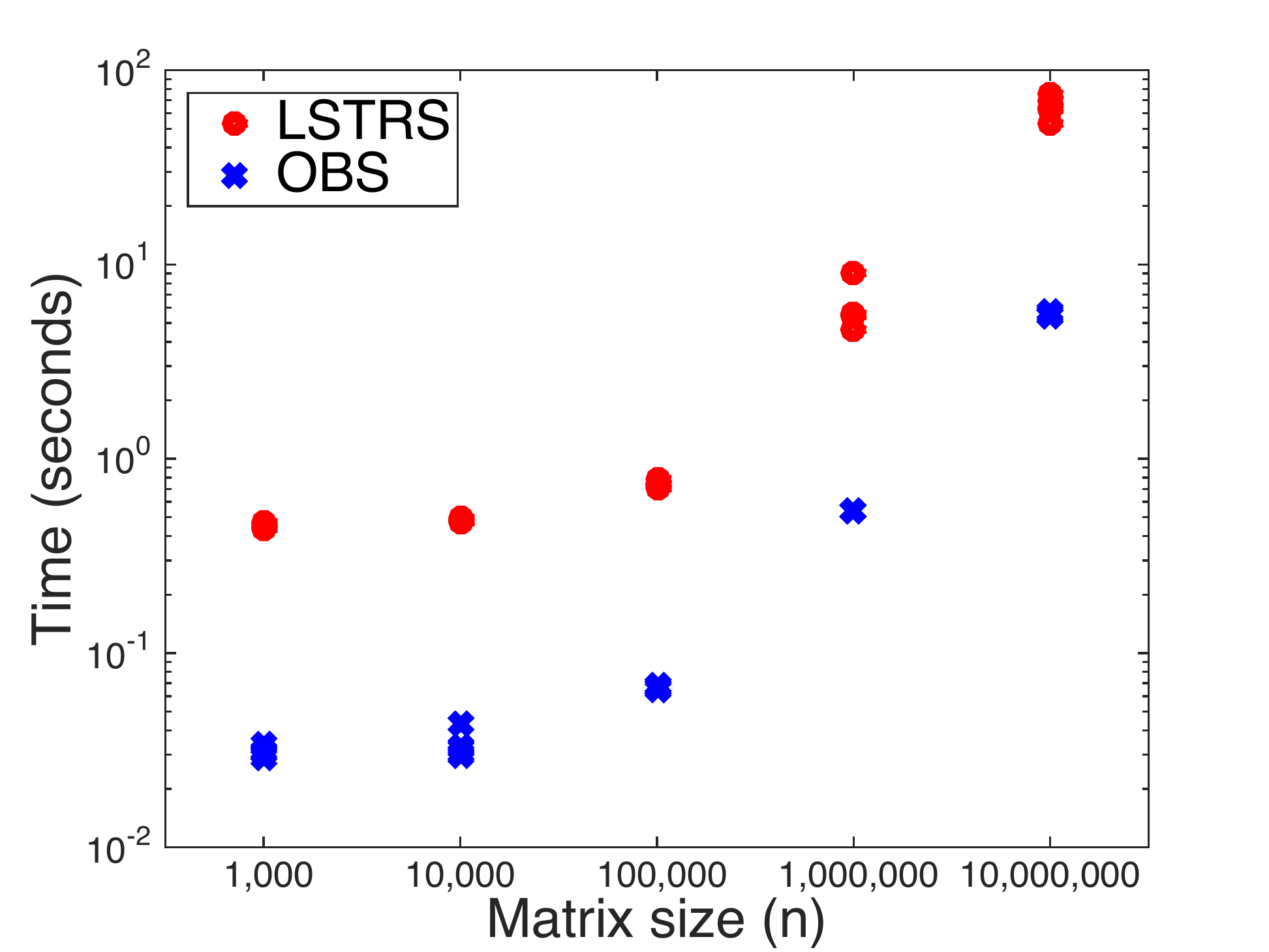} \\
Experiment 3(a) & Experiment 3(b)\\
\includegraphics[height=3.9cm]{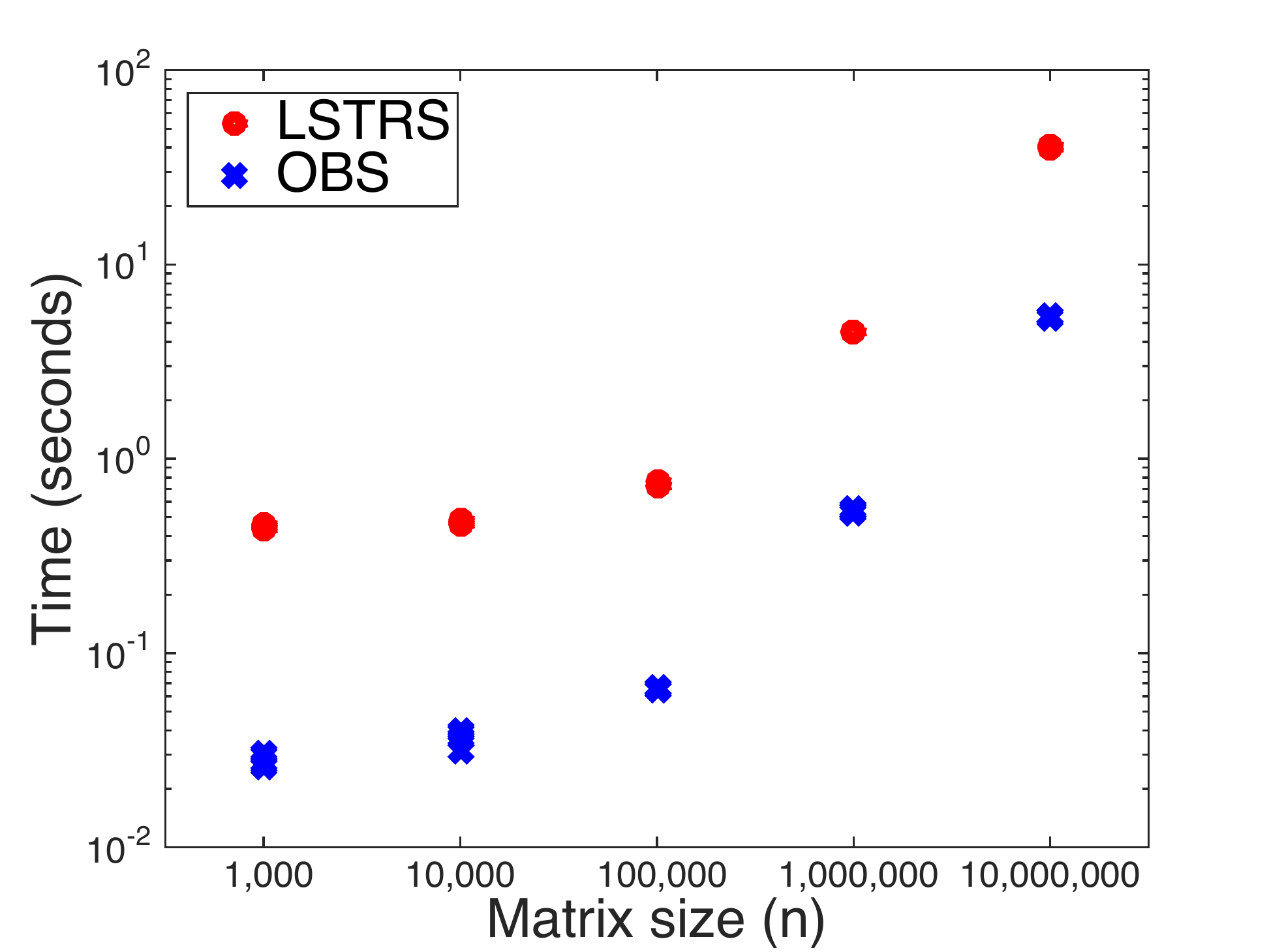} &
\includegraphics[height=3.9cm]{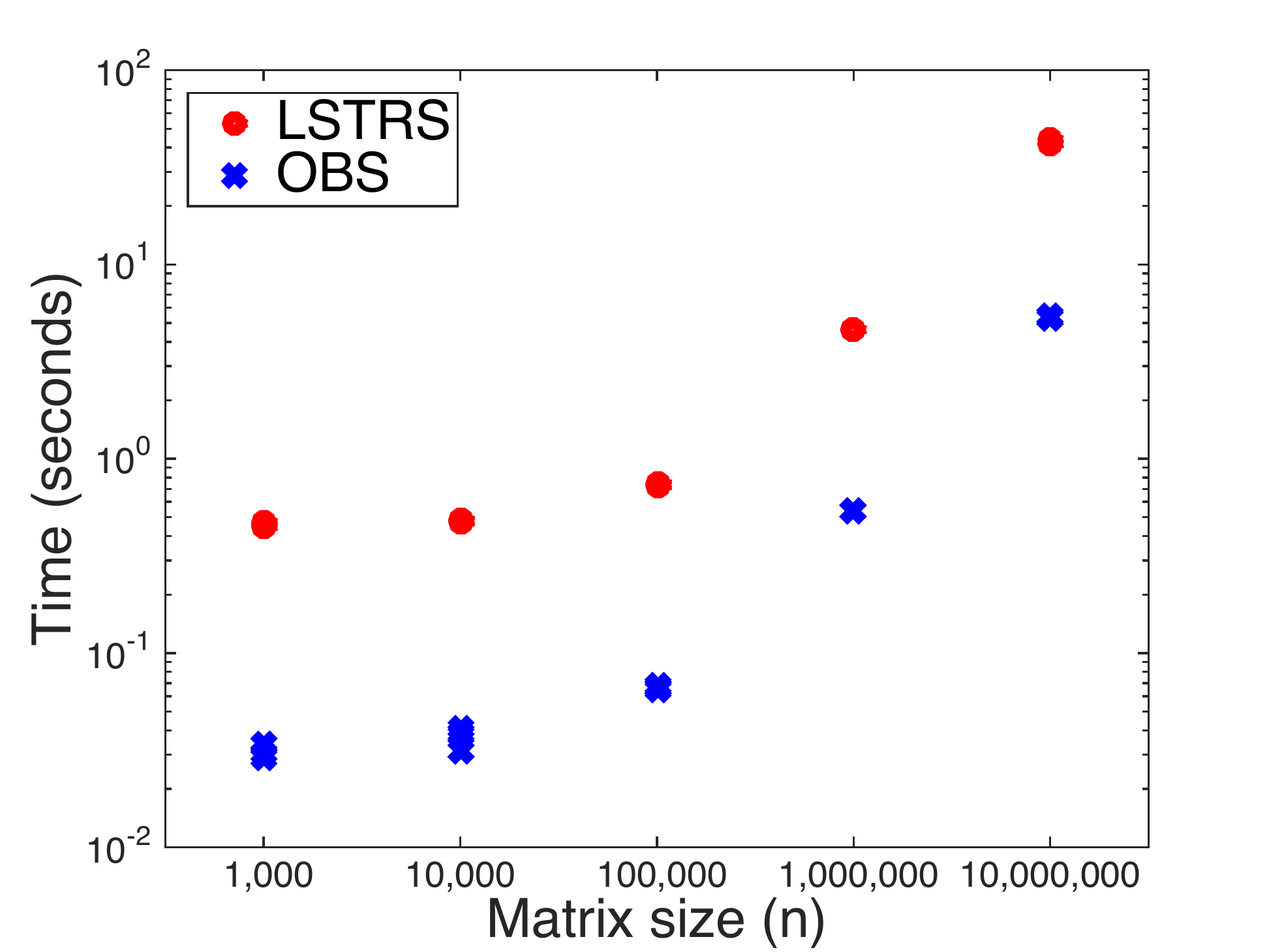} \\
Experiment 4(a) & Experiment 4(b) \\
\includegraphics[height=3.9cm]{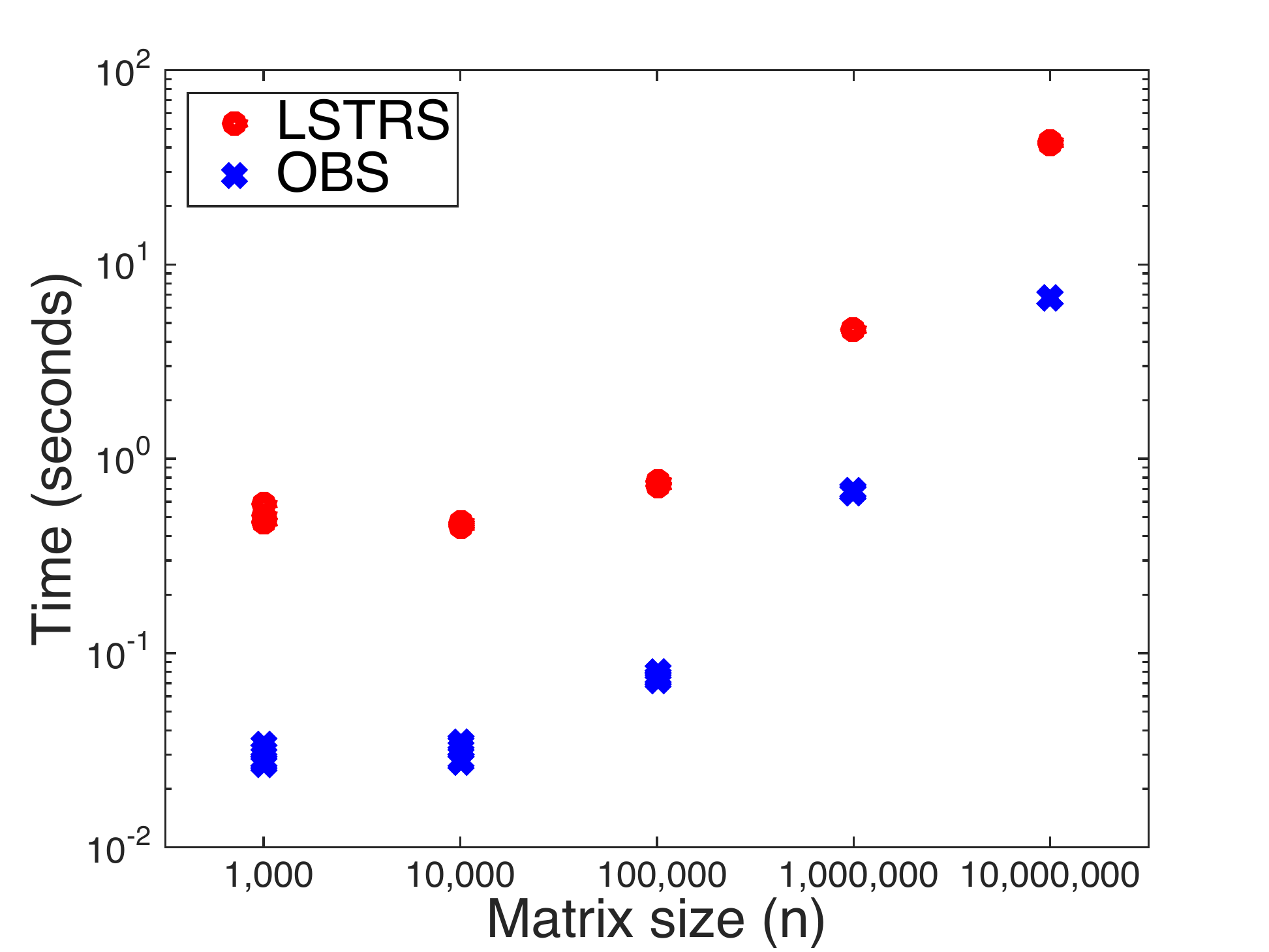} &
\includegraphics[height=3.9cm]{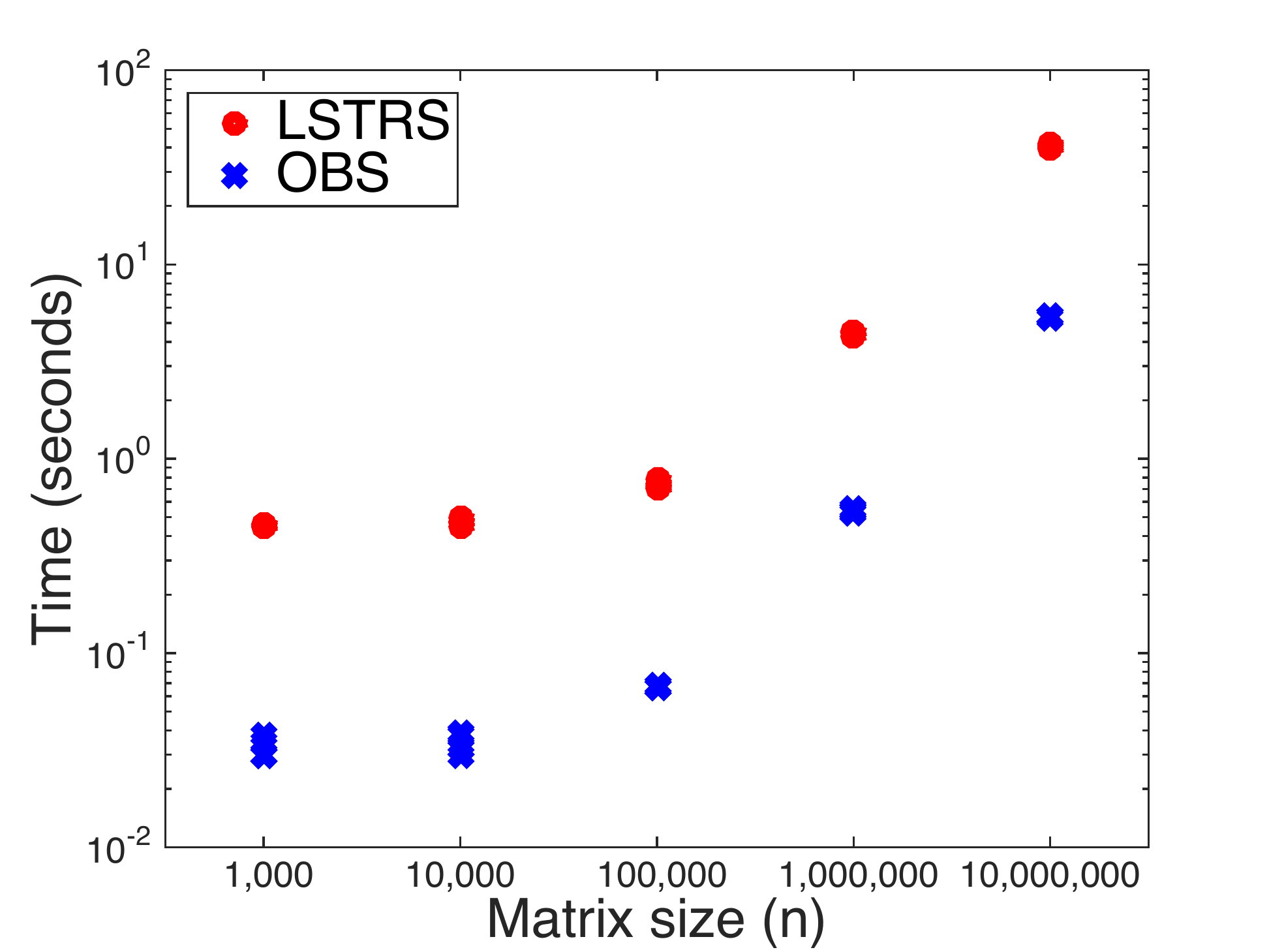} \\
Experiment 5(a) & Experiment 5(b)
\end{tabular}
\caption{Semi-log plots of the computational times (in seconds).
Each experiment was run five times; computational 
time for the LSTRS and OBS method are shown for each run. 
In all cases, the OBS method outperforms LSTRS in terms of
computational time.}
\end{figure}

\section{Concluding remarks}
In this paper, we presented the \OBS{} method, which solves trust-region
subproblems of the form (\ref{eqn-trustProblem}) where $B$ is a large \LSR{}
matrix.  The \OBS{} method uses two main strategies.  In one strategy,
$\sigma^*$ is computed from Newton's method and initialized at a point
where Newton's method is guaranteed to converge monotonically to
$\sigma^*$.  With $\sigma^*$ in hand, $p^*$ is computed directly by
formula.  For the other strategy, we propose a method that relies on an
orthonormal basis to directly compute $p^*$. (In this case, $\sigma^*$ can
be determined from the spectral decomposition of $B$.)  Numerical
experiments suggest that the \OBS{} method is able to solve large \LSR{}
trust-region subproblems to high accuracy. Moreover, the method appears
to be more robust than the \LSTRS{} method, which does not exploit
the specific structure of $B$.  In particular,
the proposed \OBS{} method achieves high accuracy in less time in all of
the experiments and in all measures of optimality than the \LSTRS{}
method.  
Future research will consider
the best implementation of the \OBS{} method in a trust-region method
(see, for example, \cite{ByrKS96}),
including initialization of $\gamma$ and rules for updating the matrices
$S$ and $Y$ containing the quasi-Newton pairs.

\section{Acknowledgments}
This research is support in part by National Science Foundation grants
CMMI-1333326 and CMMI-1334042.

\bibliographystyle{abbrv}
\bibliography{BrustErwayMarcia}

\end{document}